\documentclass{amsart}
\usepackage{graphicx}
\usepackage{amsmath}
\usepackage{amssymb}%
\usepackage{amsfonts}%
\usepackage{graphicx}
\usepackage[all]{xy}
\usepackage[mathscr]{euscript}

\newtheorem{theorem}{Theorem}[section]
\newtheorem{lemma}[theorem]{Lemma}

\theoremstyle{definition}
\newtheorem{definition}[theorem]{Definition}

\newtheorem{proposition}[theorem]{Proposition}
\newtheorem{corollary}[theorem]{Corollary}

\usepackage{graphicx}
\theoremstyle{remark}
\newtheorem{remark}[theorem]{Remark}

\numberwithin{equation}{section}
\usepackage{graphicx}

\begin{document}

\title{The Classification of Homotopy Classes of Bounded Curvature Paths}

\author{Jos\'{e} Ayala}
\address{FIA, Universidad Arturo Prat, Iquique, Chile}
\email{jayalhoff@gmail.com}

\author{Hyam Rubinstein}
\address{Department of Mathematics and Statistics, University of Melbourne
              Parkville, VIC 3010 Australia}
\email{rubin@ms.unimelb.edu.au}


%
%
%
\maketitle

\begin{abstract}

A bounded curvature path is a continuously differentiable piecewise $C^2$ path with bounded absolute curvature that connects two points in the tangent bundle of a surface. In this note we give necessary and sufficient conditions for two bounded curvature paths, defined in the Euclidean plane, to be in the same connected component while keeping the curvature bounded at every stage of the deformation. Following our work in \cite{papera}, \cite{paperb} and \cite{paperc} this work finishes a program started by Lester Dubins in \cite{dubins 2} in 1961.
\end{abstract}


\section{Introduction}

A bounded curvature path corresponds to a $C^1$ and piecewise $C^2$ path lying in ${\mathbb R}^2$ having its curvature bounded by a positive constant and connecting two elements of the tangent bundle ${T \mathbb R}^2$. Throughout the papers \cite{papera}, \cite{paperb} and \cite{paperc} we have developed techniques to answer questions about the connectivity of the spaces of bounded curvature paths and to establish the minimal length elements in these spaces. In 1961 in {\it On Plane Curves with Curvature}, Lester Dubins raised important questions initiating the study of the homotopy classes of bounded curvature paths.  {\it Here we only begin the exploration, raise some questions that we hope will prove stimulating, and invite others to discover the proofs of the definite theorems, proofs that have eluded us} (see \cite{dubins 2} page 471).  This work provides an answer to the principal question raised by Lester Dubins, that is, finding necessary and sufficient conditions for two bounded curvature paths to lie in the same homotopy class. 

 Let $\Gamma(n)$ be the space of bounded curvature paths having fixed initial and final points $\mbox{\sc x,y}\in T{\mathbb R}^2$ and having winding number $n$. We define the winding number by closing up the paths using a fixed path from the final to the initial point in $T{\mathbb R}^2$. By analysing the distances between the points {\sc x} and {\sc y} we obtain four conditions of distances called proximity conditions. The proximity conditions give qualitative insight to understand some of the topological features of $\Gamma(n)$, as for example the number of homotopy classes in $\Gamma(n)$. In the papers \cite{papera} and \cite{paperb} we classified the minimal length elements in $\Gamma(n)$ for all $n$ and all $\mbox{\sc x,y}\in T{\mathbb R}^2$. In particular, in \cite{papera} we characterised the bounded curvature paths that are candidates for being of minimum length via a normalisation process. This is then followed by a reduction process performed on bounded curvature paths to obtain the well known {\sc csc}, {\sc ccc} characterisation as first obtained in \cite{dubins 1} (compare Theorem \ref{embdub}). In \cite{paperb} we characterised the bounded curvature paths of minimal length in $\Gamma(n)$ for all $n$ (compare Theorem \ref{singudub}). This is achieved by using some of the the ideas developed in \cite{papera}. The normalisation and reduction processes are intimately related with the concept of fragmentation of a path and this concept is developed in a way that leaves the winding number of the paths invariant. However, in order to achieve the desired classification, we need to develop a continuity argument keeping track of the curvature bound at any stage of the deformation. These ideas are crucial for our purposes and are developed in Section 3.

In \cite{paperc} we analysed the extent to which paths in $\Gamma(n)$ can be made homotopic via a one parameter family of bounded curvature paths, paying special attention on the proximity condition the points {\sc x} and {\sc y} satisfy. We concluded that under certain conditions it is not possible to continuously deform embedded bounded curvature paths lying in a planar region $\Omega$ to bounded curvature paths lying outside $\Omega$ while keeping the curvature bounded at each state of the deformation  (compare Corollary \ref{mainresultp1}). In particular, the region $\Omega$ is compact and every path in the homotopy class of the minimal length path cannot be homotopic via a one parameter family of bounded curvature paths to a path outside of $\Omega$. 

The ideas described in the previous paragraphs will be essential for giving necessary and sufficient conditions for two bounded curvature paths, defined in the euclidean plane, to be in the same connected component while keeping the curvature bounded at every stage of the deformation. Our main result Theorem  \ref{classificationbcps} can be described informally as follows. Given two directed points $\mbox{\sc x,y}\in T{\mathbb R}^2$, the collection of homotopy classes of bounded curvature paths, starting and ending at {\sc x}, {\sc y}, with winding number $k$, has either one or two elements.  The case of two homotopy classes only occurs for a unique choice of $k$ and for {\sc x}, {\sc y} satisfying a proximity condition D. This condition D is that the points corresponding to $x,y\in {\mathbb R}^2$ are distance less than 4 apart, assuming that the curvature of the paths is bounded by 1. Following our work in \cite{papera}, \cite{paperb} and \cite{paperc} this paper finishes a program started by Lester Dubins in \cite{dubins 2} in 1961.

\section{Preliminaries}
 
Let us denote by $T{\mathbb R}^2$ the tangent bundle of ${\mathbb R}^2$. The elements in $T{\mathbb R}^2$ correspond to pairs $(x,X)$ sometimes denoted just by {\sc x}. As usual, the first coordinate corresponds to a point in ${\mathbb R}^2$ and the second to a tangent vector to ${\mathbb R}^2$ at $x$.

\begin{definition} \label{adm_pat} Given $(x,X),(y,Y) \in T{\mathbb R}^2$, we say that a path $\gamma: [0,s]\rightarrow {\mathbb R}^2$ connecting these points is a {\it bounded curvature path} if:
\end{definition}
 \begin{itemize}
\item $\gamma$ is $C^1$ and piecewise $C^2$.
\item $\gamma$ is parametrized by arc length (i.e $||\gamma'(t)||=1$ for all $t\in [0,s]$).
\item $\gamma(0)=x$,  $\gamma'(0)=X$;  $\gamma(s)=y$,  $\gamma'(s)=Y.$
\item $||\gamma''(t)||\leq \kappa$, for all $t\in [0,s]$ when defined, $\kappa>0$ a constant.
\end{itemize}
Of course, $s$ is the arc-length of $\gamma$.

  The first condition means that a  bounded curvature path has continuous first derivative and piecewise continuous second derivative. For the third condition, without loss of generality, we can extend the domain of $\gamma$ to $(-\epsilon,s+\epsilon)$ for $\epsilon$ arbitrarily small. Sometimes we describe the third item as the endpoint condition. The last condition means that  bounded curvature paths have absolute curvature bounded above by a positive constant. We denote the interval $[0,s]$ by $I$.
  Also, when more than one path is under consideration, we write $\gamma: [0,s_\gamma]\rightarrow {\mathbb R}^2$ to specify arc-length.

\begin{definition} \label{admsp} Given $\mbox{\sc x,y}\in T{\mathbb R}^2$ and a maximum curvature $\kappa>0$. The space of bounded curvature paths satisfying the given endpoint condition is denoted by $\Gamma(\mbox{\sc x,y})$. \end{definition}

It is important to note that the topological and geometrical properties of the space of bounded curvature paths $\Gamma(\mbox{\sc x,y})$ depends on the chosen elements in $T{\mathbb R}^2$. Properties such as compactness, connectedness, as well as the type of minimal length elements in $\Gamma(\mbox{\sc x,y})$, are intimately related with the  endpoint condition.

\begin{remark}\label{coord} In Definition \ref{adm_pat} the curvature is bounded above by a positive constant. This constant can be normalized by choosing suitable scaling via a dilation (contraction) of the plane. Throughout this work we will always consider $\kappa=1$ in Definition \ref{adm_pat}. Moreover, without loss of generality, consider the origin of our orthogonal global coordinate system as the base point $x$ and let $X$ lie in the space generated by the vector $(\frac{\partial}{\partial x})$ in the tangent space $T_x{\mathbb R}^2$ denoted, from now, $\mathbb R^2$.
\end{remark}

  \begin{definition} Let $\mbox{\sc C}_ l(\mbox{\sc x})$ be the unit circle tangent to $x$ and to the left of $X$. An analogous interpretation applies for $\mbox{\sc C}_ r(\mbox{\sc x})$, $\mbox{\sc C}_ l(\mbox{\sc y})$ and $\mbox{\sc C}_ r(\mbox{\sc y})$. These circles are called {\it adjacent circles} and their arcs are called {\it adjacent arcs}. We denote their centers with lower-case letters, so the center of $\mbox{\sc C}_ l(\mbox{\sc x})$ is denoted by $c_l(\mbox{\sc x})$.
  \end{definition}

\begin{remark}\label{bcurvhomot}  We employ the following convention: When a path is continuously deformed under parameter $p$, we reparametrize each of the deformed paths by its arc-length. Thus $\gamma: [0,s_p]\rightarrow {\mathbb R}^2$ describes a deformed path at parameter $p$, with $s_p$ corresponding to its arc-length.
\end{remark}

\begin{definition}  \label{hom_adm} Given $\gamma,\eta \in \Gamma(\mbox{\sc x,y})$. A {\it bounded curvature homotopy}  between $\gamma: [0,s_0] \rightarrow \mathbb R^2$ and $\eta: [0,s_1] \rightarrow \mathbb R^2$ corresponds to a continuous one-parameter family of immersed paths $ {\mathcal H}_t: [0,1] \rightarrow \Gamma(\mbox{\sc x,y})$ such that:
\begin{itemize}
\item ${\mathcal H}_t(p): [0,s_p] \rightarrow \mathbb R^2$ for $t\in [0,s_p]$ is an element of $\Gamma(\mbox{\sc x,y})$ for all $p\in [0,1]$.
\item $ {\mathcal H}_t(0)=\gamma(t)$ for $t\in [0,s_0]$ and ${\mathcal H}_t(1)=\eta(t)$ for $t\in [0,s_1]$.
\end{itemize}
\end{definition}

It is of interest to establish under what conditions a bounded curvature path can be deformed, while keeping all the intermediate paths in $\Gamma(\mbox{\sc x,y})$, to a path of arbitrarily large length.

\begin{definition}  A bounded curvature path is said to be {\it free} if there exists a bounded curvature path arbitrarily long and a bounded curvature homotopy deforming one path into another.
\end{definition}

The next remark summarizes well known facts about homotopy classes of paths on metric spaces. These facts are naturally adapted for elements in $\Gamma(\mbox{\sc x,y})$ for all endpoint conditions. Our convention is to consider $\Gamma(\mbox{\sc x,y})$ together with the $C^1$ metric.

\begin{remark}  For $\mbox{\sc x,y}\in T{\mathbb R}^2$ then,
\end{remark}
\begin{itemize}
\item Two bounded curvature paths are {\it bounded-homotopic} if there exists a bounded curvature homotopy from one path to another.  Such a relation is an equivalence relation.
\item A homotopy class on $\Gamma(\mbox{\sc x,y})$ corresponds to an equivalence class on $\Gamma(\mbox{\sc x,y})$ under the relation described above.
\item The maximal path connected sets of $\Gamma(\mbox{\sc x,y})$ are called the {\it homotopy classes}  (or {\it path connected components}) of $\Gamma(\mbox{\sc x,y})$.
\item The homotopy classes of $\Gamma(\mbox{\sc x,y})$ are nonempty, pairwise disjoint and their union is $\Gamma(\mbox{\sc x,y})$.
\end{itemize}



In the next subsection we collect definitions and results from \cite{paperc}.
\subsection{Spaces of Bounded Curvature Paths}\label{proximity}

 For given $\mbox{\sc x,y}\in T{\mathbb R}^2$, the normalized bounded curvature constraint $|\kappa|\leq 1$ makes the four associated adjacent circles as barriers for local deformations of paths in $\Gamma(\mbox{\sc x,y})$ around $x$ and $y$. We extract four simple pairs of inequalities. These summarize the possible relations of distance between the adjacent circles. We denote by $d$ the Euclidean metric in $\mathbb R^2$. It is easy to see that $d(c_l(\mbox{\sc x}),c_r(\mbox{\sc x}))=2$ and that $d(c_l(\mbox{\sc y}),c_r(\mbox{\sc y}))=2$.

Let an endpoint condition $\mbox{\sc x,y}\in T{\mathbb R}^2$ be given. Taking into account the centres of the associated adjacent circles, four conditions of distance between them can be stated:

 \begin{equation} d(c_l(\mbox{\sc x}),c_l(\mbox{\sc y}))\geq 4 \quad \mbox{and}\quad d(c_r(\mbox{\sc x}),c_r(\mbox{\sc y}))\geq4 \label{con_a}\tag{i}\end{equation}
 \begin{equation} d(c_l(\mbox{\sc x}),c_l(\mbox{\sc y}))< 4 \quad \mbox{and}\quad d(c_r(\mbox{\sc x}),c_r(\mbox{\sc y}))\geq 4 \label{con_b}\tag{ii} \end{equation}
  \begin{equation} d(c_l(\mbox{\sc x}),c_l(\mbox{\sc y}))\geq4 \quad \mbox{and}\quad d(c_r(\mbox{\sc x}),c_r(\mbox{\sc y}))< 4  \label{con_b'}\tag{iii} \end{equation}
   \begin{equation} d(c_l(\mbox{\sc x}),c_l(\mbox{\sc y}))< 4 \quad \mbox{and}\quad d(c_r(\mbox{\sc x}),c_r(\mbox{\sc y}))< 4 \label{con_c}\tag{iv} \end{equation}
   \vspace{.3cm}

 These possible configurations for the centers of the adjacent circles $\mbox{\sc C}_ l(\mbox{\sc x})$, $\mbox{\sc C}_ r(\mbox{\sc x})$, $\mbox{\sc C}_ l(\mbox{\sc y})$, and $\mbox{\sc C}_ r(\mbox{\sc y})$ are illustrated in Figures~\ref{figab} and {\ref{figc}}. Observe that  (ii) and (iii) give the same planar configurations via a reflection (see Figure \ref{figab} bottom). In addition, (iv) implies that $d(x,y)<4$.

{ \begin{figure} [[htbp]
 \begin{center}
\includegraphics[width=.9\textwidth,angle=0]{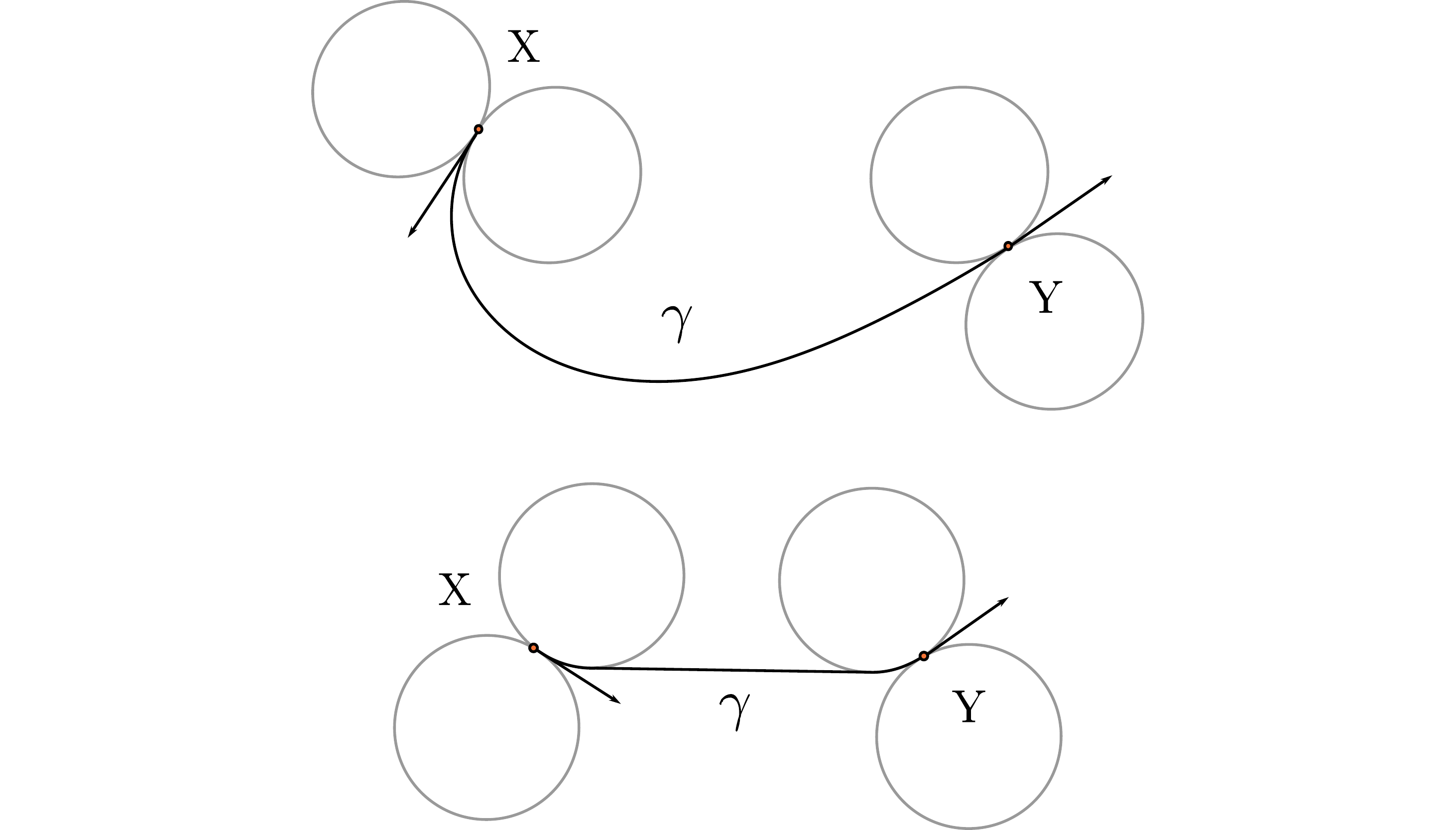}
\end{center}
\caption{Top: An example of configuration given by condition (i). Bottom: An example of configuration given by condition (ii)}
\label{figab}
\end{figure}}

\begin{remark} \label{condcconf} If the given endpoint condition $\mbox{\sc x,y}\in T{\mathbb R}^2$ satisfies proximity condition (iv) we have three possible scenarios (see Figure \ref{figc}):
\begin{itemize}
\item $\mbox{\sc x,y}\in T{\mathbb R}^2$ is the endpoint condition of a path consisting of a single arc of a unit circle of length less than $\pi$ or $\mbox{\sc x,y}\in T{\mathbb R}^2$ is the endpoint condition of a path consisting of a concatenation of two arcs of unit circles each of length less than $\pi$.
\item The curvature bound for elements in $\Gamma(\mbox{\sc x,y})$ induces the existence of a region denoted by $\Omega$ (see Figure \ref{figc} left).
\item $\mbox{\sc x,y}\in T{\mathbb R}^2$ is the endpoint condition of a path bounded-homotopic to a path of arbitrary length.
\end{itemize}
\end{remark}

{{ \begin{figure} [[htbp]
 \begin{center}
\includegraphics[width=1\textwidth,angle=0]{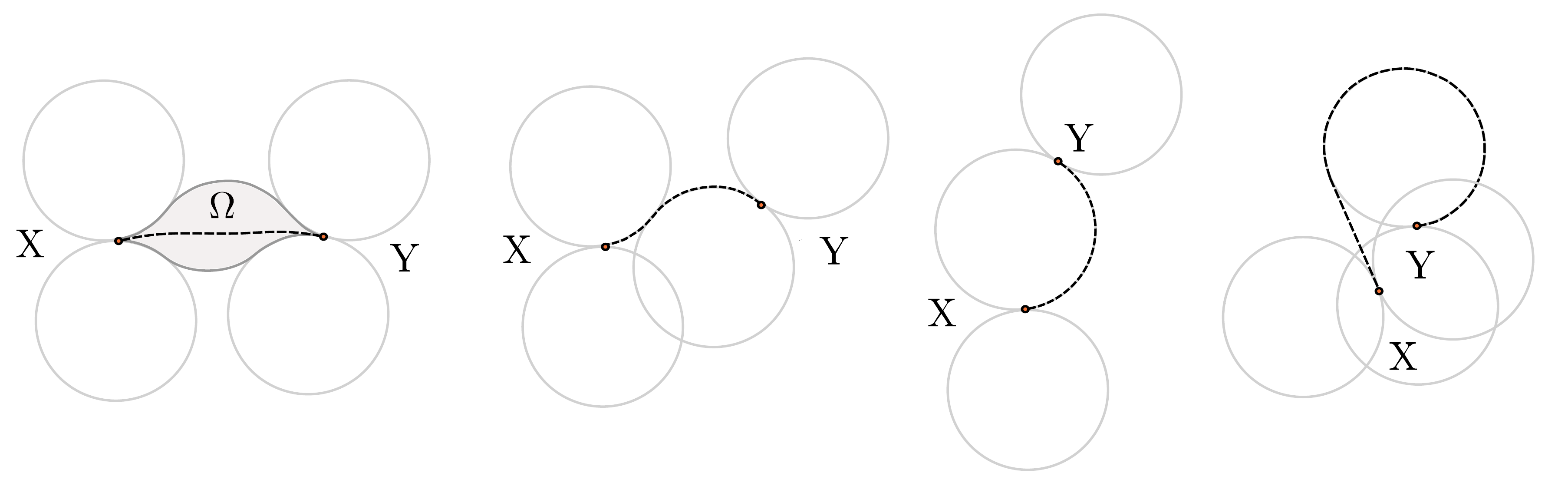}
\end{center}
\caption{The possible scenarios under proximity condition (iv). The dashed trace represents a path in $\Gamma(\mbox{\sc x,y})$.}
 \label{figc}
\end{figure}}}

In order to make the figures clearer, sometimes we omit the initial and final tangent vectors.

\begin{definition} The space $\Gamma({\mbox{\sc x,y}})$ satisfies proximity condition {\sc A} if its endpoint condition satisfies (i).
\end{definition}

\begin{definition}  The space $\Gamma({\mbox{\sc x,y}})$ satisfies proximity condition {\sc B} if its endpoint condition satisfies (ii) or (iii).
\end{definition}

\begin{definition}\label{condcnoreg}  The space $\Gamma({\mbox{\sc x,y}})$ satisfies proximity condition {\sc C} if its endpoint condition satisfies (iv) and $\Gamma({\mbox{\sc x,y}})$ contains a path that has as a subpath:
\begin{itemize}
\item an arc of circle of length greater than or equal to $\pi $ or
\item a line segment of length greater than or equal to $4$.
\end{itemize}
\end{definition}

The next definition is of special importance, it captures features that allow the existence of non-trivial homotopy classes of bounded curvature paths in $\Gamma(\mbox{\sc x,y})$. For details and related consequences on the existence of non trivial homotopy classes of bounded curvature paths refer to \cite{paperc}.

\begin{definition}\label{defd} The space $\Gamma({\mbox{\sc x,y}})$ satisfies proximity condition {\sc D} if its endpoint condition satisfies (iv) and:
\begin{itemize}
\item The curvature bound for elements in $\Gamma(\mbox{\sc x,y})$ induces the existence of a region $\Omega$ or
\item $\Gamma(\mbox{\sc x,y})$ contains a path which is an arc of a unit circle of length less than $\pi $ or
\item $\Gamma(\mbox{\sc x,y})$ contains a path which is a concatenation of two oppositely oriented arcs of unit circles of length less than $\pi $ each.
\end{itemize}
\end{definition}

The three items in Definition \ref{defd} are mutually exclusive. For example, it's easy to see that no endpoint condition simultaneously allow the existence of a path which is an arc of a unit circle of length less than $\pi $ and a path which is a concatenation of two oppositely oriented arcs of unit circles of length less than $\pi $ each.

\begin{definition} \hspace{5cm}
\begin{itemize}
\item A path $\gamma: I \rightarrow {\mathbb R}^2$ is {\it in} $\Omega$ if $\gamma(t) \in \Omega$ for all $t\in I$.
\item A path $\gamma: I \rightarrow \mathbb R^2$ is {\it not in} $\Omega$ if there exists $t\in I$ such that $\gamma(t) \notin \Omega$.
\end{itemize}
\end{definition}

\begin{definition} We denote by $\Delta(\Omega)\subset \Gamma(\mbox{\sc x,y})$ the space of embedded bounded curvature paths in $\Omega$. And, by $\Delta'(\Omega)\subset \Gamma(\mbox{\sc x,y})$ the space of bounded curvature paths not in $\Omega$.
\end{definition}

 It was a relatively long standing problem to find a mathematical argument to prove that embedded bounded curvature paths {\it in} $\Omega$ cannot be made bounded-homotopic to paths not in $\Omega$ compare [7]. In [4] we developed techniques to prove the following results.

\begin{theorem}\label{insbound} (cf. Theorem 8.12 in \cite{paperc}) Paths in $\Delta(\Omega)$ have bounded length.
\end{theorem}

\begin{corollary}\label{mainresultp1} (cf. Corollary 9.2 in \cite{paperc}) Suppose the endpoint condition $\mbox{\sc x,y}\in T{\mathbb R}^2$ carries a region $\Omega$. Then the space $\Gamma(\mbox{\sc x,y)}$ is partitioned into the spaces ${\Delta}(\Omega)$ and ${\Delta}'(\Omega)$. That is, ${\Delta}(\Omega)$ and ${\Delta}'(\Omega)$ belong to different homotopy classes. In particular the elements in ${\Delta}(\Omega)$ are not free paths.
\end{corollary}

\begin{corollary}\label{cannotsing}(cf. Corollary 9.5 in \cite{paperc}) Paths in $\Delta(\Omega)$ cannot be made bounded-homotopic to a bounded curvature paths with self intersections.
\end{corollary}

\subsection{cs Paths}
A particularly interesting type of bounded curvature path is the so called {\it cs} path. These {\it cs} paths correspond to multiple concatenations of arcs of a unit circle of nonzero length and line segments. Bounded curvature paths of {\it cs} type are simple since their curvature is constant almost everywhere; such a fact makes the construction of  bounded curvature homotopies easy.

\begin{definition} A path is in {\it cs} {\it form} if it corresponds to a finite number of concatenations of arcs of unit radius circles of nonzero length and line segments. The number of line segments plus the number of circular arcs is called the {\it complexity} of the path. Paths in {\it cs} form are called {\it cs} paths.
\end{definition}

\begin{definition} A {\sc csc} path corresponds to a concatenation of an arc of unit radius and length less than $2\pi$, followed by a line segment, followed by an arc of unit radius and length less than $2\pi$. The three concatenated parts are called {\it components} of the path. Taking into account the path orientation ({\sc r} denotes a clockwise traversed arc and {\sc l} denotes a counterclockwise traversed arc), then the {\sc csc} paths can be presented in four possible words given by {\sc lsl}, {\sc lsr}, {\sc rsr} and {\sc rsl}.
\end{definition}

When possible, it is convenient to work with $cs$ paths since we can perform some deformations on the latter in a way that it is clear that the resulting path is not only a bounded curvature path but also a $cs$ path. Next we introduce a set of operations of $cs$ paths.

\paragraph{Operations of Type I}
In order to perform {\it operations of type I}, we consider a point in the $cs$ path as the rotation axis. Once the rotation axis point is chosen, we twist the pieces of the path on opposite sides of the rotation axis point, clockwise or counterclockwise until we obtain a $cs$ path containing two oppositely oriented loops making a {\it figure 8} shape, (see Figure \ref{figopi}).

{ \begin{figure} [[htbp]
 \begin{center}
\includegraphics[width=0.8\textwidth,angle=0]{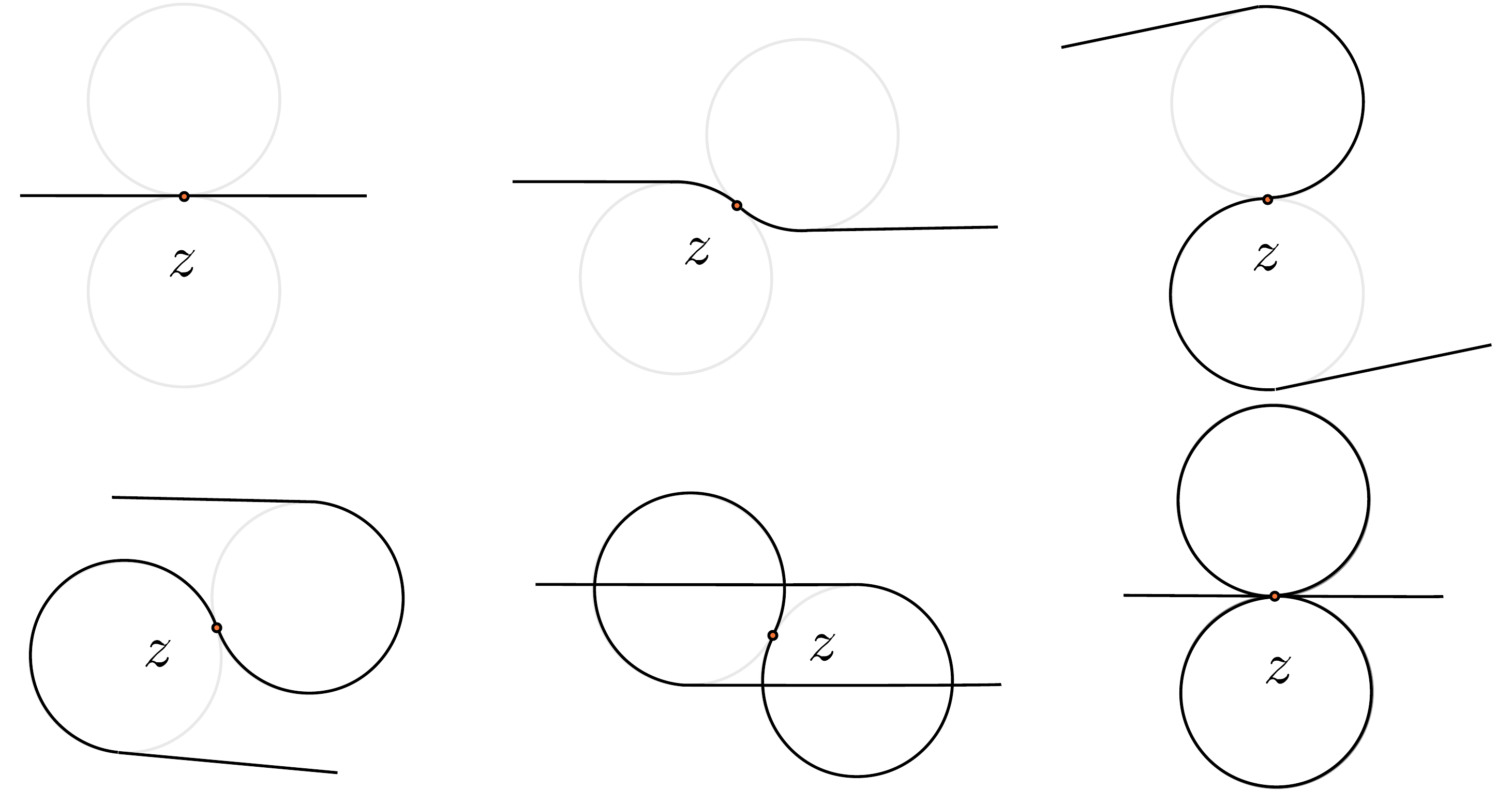}
\end{center}
\caption{A sequence of moves describing an operation of type {\it I} with rotation axis point $z$.}
\label{figopi}
\end{figure}}

\paragraph{Operations of Type II} We apply {\it operations of type II} to $cs$ paths in between {\sc csc} components satisfying proximity conditions A, B or C. We perform the operation by pushing the line segment component (up or down) with a disk as shown in Figure \ref{figopii}.

{ \begin{figure} [[htbp]
 \begin{center}
\includegraphics[width=0.6\textwidth,angle=0]{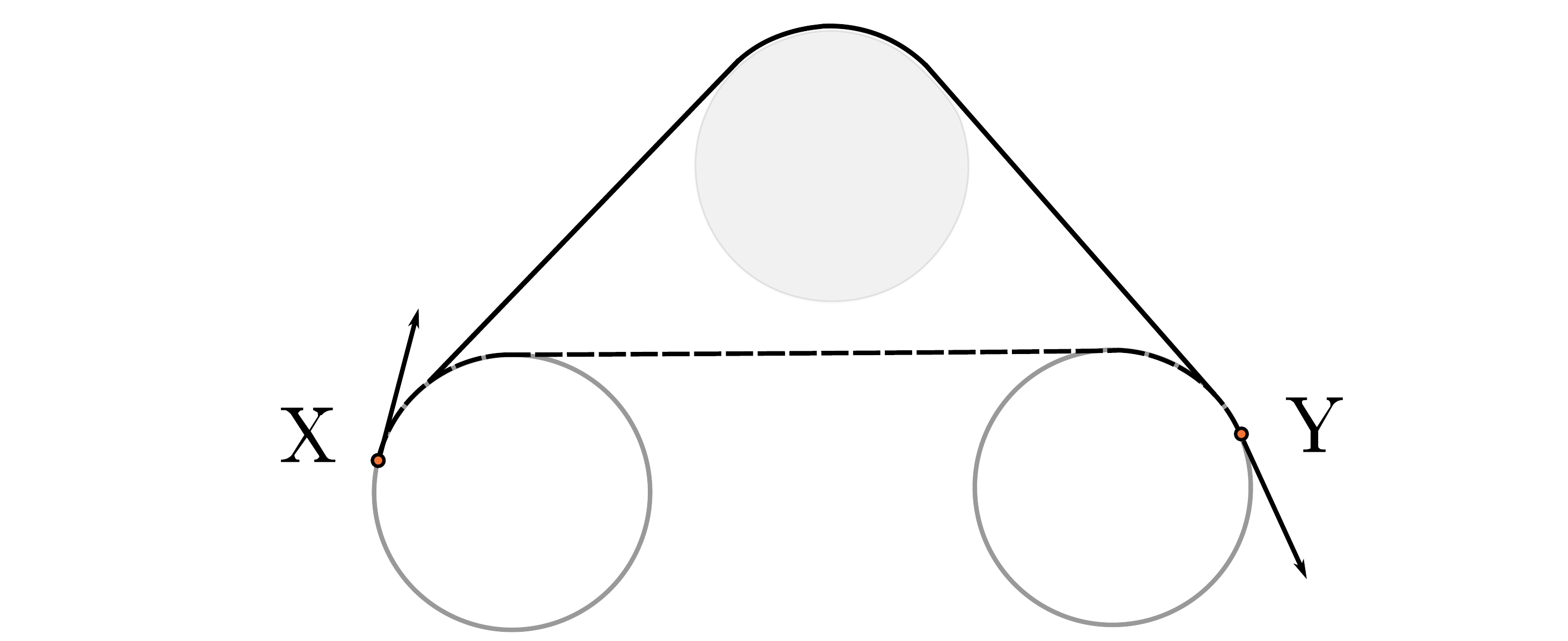}
\end{center}
\caption{The dashed path is continuously deformed to a {\sc cscsc} path.}
 \label{figopii}
\end{figure}}

 \subsection{The Winding Number of Paths} \label{windingnumberofpaths}

In \cite{paperb} we introduced the concept of winding number for elements in $\Gamma(\mbox{\sc x,y})$. In this subsection we refresh some of the terminology needed in this work. Refer to \cite{paperb} for details about this topic.

Consider the exponential map $\exp: {\mathbb R} \rightarrow {\mathbb S}^1$.

\begin{definition} For a path $\gamma:I\to{\mathbb R}^2$. The {\it turning map} $\tau$ is defined in the following diagram,
\[ \xymatrix{ I  \ar[d]_{\tau}   \ar@{>}[dr]^{w} &  \\
                     {\mathbb R} \ar[r]_{\exp}  & {\mathbb S}^1} \]

The map $w:I \rightarrow {\mathbb S}^1$ is called the {\it direction map}, and gives the derivative $\gamma'(t)$ of the path $\gamma$ at $t\in I$. The turning map $\tau:I\rightarrow {\mathbb R}$ gives the turning angle the derivative vector makes at $t\in I$ with respect to $\exp(0)$ i.e., the turning angle $\gamma'(t)$ makes with respect to the $x$-axis. \end{definition}

The following definition is made to obtain an integer invariant for bounded curvature paths that are not necessarily closed.

\begin{definition}The {\it relative winding number} of a bounded curvature path $\gamma : [0,s]\rightarrow {\mathbb R}^2$ is denoted by $\rho(\gamma)$, and is defined by:

 \begin{equation} \rho (\gamma) = \frac{ \tau(s)-z }{2\pi} \label{relwin}\end{equation}

where $z\in(- \pi, \pi]$ satisfies $\exp(z)=\gamma'(s)$.
\end{definition}

Thus the relative winding number of a bounded curvature path corresponds to the number of turns the derivative vector of a path makes as it travels along the path.

 \begin{remark}\label{convlambda} A bounded curvature path $\gamma:[0,s_\gamma]\rightarrow {\mathbb R}^2$ such that $x \neq y$ can be made into a closed path (after a reparametrisation) by concatenating it with a path $ \lambda:[0,s_\lambda]\rightarrow {\mathbb R}^2$ such that:
  \begin{itemize}
 \item $ \lambda(0)=\gamma(s_\gamma)$ and $ \lambda'(0)=\gamma'(s_\gamma)$.
 \item $ \lambda(s_\lambda)=\gamma(0)$ and $ \lambda'(s_\lambda)=\gamma'(0)$.
 \item $\lambda$ is a minimal length bounded curvature path between $(y,Y)$ and  $(x,X)$.
 \item If $\lambda$ is not unique then we choose a preferred one and fix it.
 \end{itemize}
The path $\lambda$ is called the {\it closure path}.
  \end{remark}

In analogy with Whitney's work in \cite{whitney} we introduce the following terminology.

 \begin{definition} For a path $\gamma:I\rightarrow {\mathbb R}^2$, the {\it winding number} with respect to the prescribed closure path $ \lambda$ is defined by:
  \begin{equation} W_\lambda (\gamma)=\frac{\tau(s_\gamma)+\tau(s_ \lambda)}{2 \pi} \label{winumb}\end{equation}
 \end{definition}

\begin{remark}\label{chi}
A theorem of Whitney, see \cite{whitney}, establishes that any curve can be perturbed to have a finite number of transversal self intersections. The number of transversal self intersections of $\gamma$ can be chosen to be minimal and may be denoted by $\chi$. In this work we will develop a continuity argument preserving the bounded curvature property for homotopies between general bounded curvature paths. Then we can adapt Whitney's arguments in \cite{whitney} for bounded curvature paths.\end{remark}

  \begin{definition} Given $\mbox{\sc x,y}\in T{\mathbb R}^2$ together with a prescribed closure path $ \lambda$. The space of bounded curvature paths satisfying the given endpoint condition and having winding number $n$ is denoted by:
$${\Gamma}(n)=\{\gamma \in {\Gamma}(\mbox{\sc x,y}) \,|\,\,\, W_ \lambda(\gamma)=n,\,\,\,n\in {\mathbb Z}\}.$$
Here a closure path is chosen once and for all $n\in {\mathbb Z}$.
We sometimes write ${\Gamma}(\mbox{\rm P},n)$ to specify that the given endpoint condition associated with the spaces $\Gamma(n)$ satisfy proximity condition {\rm P}.
\end{definition}

Note that the orientation of $\lambda$ induces an orientation of the elements on $\Gamma(\mbox{\sc x,y})$. We will consider clockwise orientation as positive and counterclockwise as negative.

\begin{definition}\label{nk} Denote by $\Gamma(k)$ the space containing the global minimum of length in $\Gamma(\mbox{\sc x,y})$.
\end{definition}

The Graustein-Whitney theorem states that two planar closed curves with the same winding number may be deformed one into another by a regular homotopy (homotopy through immersions). In addition, we cannot deform two closed curves of distinct winding numbers one into another by such a homotopy. As a consequence the winding number is a topological invariant for planar curves under regular homotopies, see \cite{whitney}. So, if two paths have different winding numbers, they are in different homotopy classes. Therefore we immediately have from the Graustein-Whitney theorem:

\begin{corollary} \label{wcpm} For every $\mbox{\sc x,y}\in T{\mathbb R}^2$ we have that ${\Gamma}(m)\cap {\Gamma}(n)=\emptyset$ for $m\neq n$.
\end{corollary}

Next we collect definitions and results in \cite{papera}.

\subsection{Normalization of Bounded Curvature Paths}

 The main goal in \cite{papera} was to answer the following question. Given $\mbox{\sc x,y}\in T{\mathbb R}^2$. Characterize the minimal length elements in $\Gamma(\mbox{\sc x,y})$. To this end we developed a normalization procedure applied to general bounded curvature paths. The idea is to subdivide a given bounded curvature path into pieces such that a path replacement performed on such pieces only leads {\sc csc} paths. This process applied to each piece of the path gives a $cs$ path. Then we concluded that the minimal length element in $\Gamma(\mbox{\sc x,y})$ is a $cs$ path of at most complexity three. In this section we invoke results, definitions and techniques introduced in \cite{papera}.

In the context of the next  definition we use the following notation. For a path $\gamma:I \rightarrow {\mathbb R}^2$ we denote its length by ${\mathcal L}(\gamma)$. The length of $\gamma$ restricted to $[a,b]\subset I$ is denoted by ${\mathcal L}(\gamma,a,b)$.

\begin{definition}\label{frag} A {\it fragmentation} of a bounded curvature path $\gamma:I \rightarrow {\mathbb R}^2$ corresponds to a finite sequence $0=t_0<t_1\ldots <t_m=s$ of elements in $I$ such that,
$${\mathcal L}(\gamma,t_{i-1},t_i)<  1 $$
with,
$$\sum_{i=1}^m {\mathcal L}(\gamma,t_{i-1},t_i) =s$$
We denote by a {\it fragment}, the restriction of $\gamma$ to the interval determined by two consecutive elements in the fragmentation.

\end{definition}

\begin{remark}\label{fragd} Is immediate from the Definition \ref{defd} and Definition \ref{frag} that fragments satisfy proximity condition {\sc D}.
\end{remark}

\begin{definition}\label{rzlip} Let ${\mbox{\sc z}}\in T{\mathbb R}^2$ with $\mbox{\sc z}=(z,Z)$ and first component corresponding to the origin of a local coordinate system with abscissa $x$ and ordinate $y$ and $Z$ being horizontal in the positive $x$ direction. Denote by ${\mathcal R}(\mbox{\sc z})$ the region enclosed by the complement of the union of the interior of the disks with boundary $\mbox{\sc C}_l(\mbox{\sc z})$ and $\mbox{\sc C}_r(\mbox{\sc z})$ intersected with a unit radius disk centered at $z$.
\end{definition}

The next result gives a natural region where every bounded curvature path is locally confined.

\begin{proposition} \label{r1r3pr} (cf. Proposition 2.7 in \cite{papera}) Any fragment $\gamma$ with $\gamma(t)=z$ and $\gamma'(t)=Z$ is contained in ${\mathcal R}(\mbox{\sc z})$ (see Figure \ref{figleavreg}).
\end{proposition}

{ \begin{figure} [[htbp]
 \begin{center}
\includegraphics[width=0.5\textwidth,angle=0]{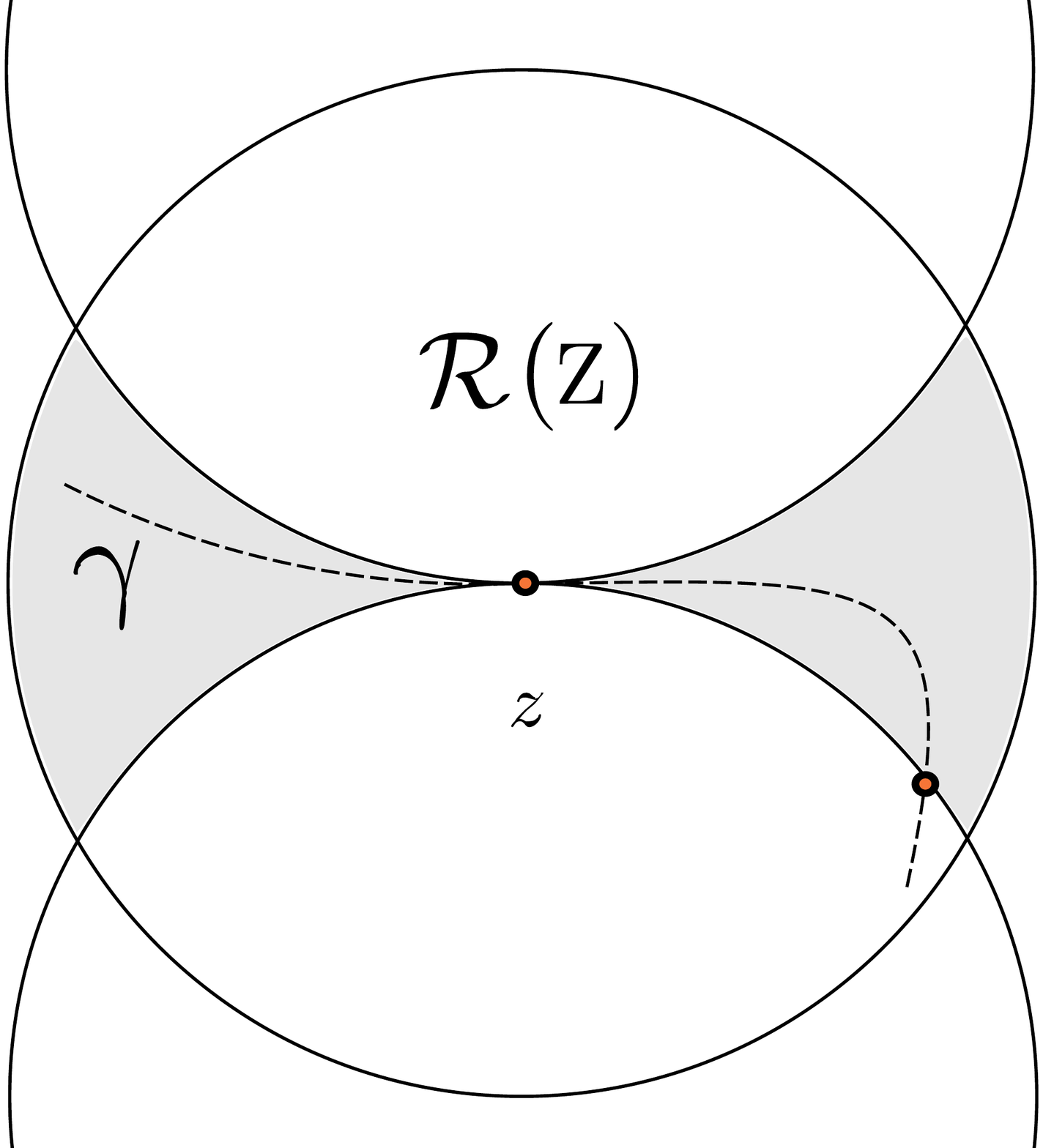}
\end{center}
\caption{A fragment never leaves ${\mathcal R}(\mbox{\sc z})$ through $\mbox{\sc C}_l(\mbox{\sc z})$ or $\mbox{\sc C}_r(\mbox{\sc z})$.}
\label{figleavreg}
\end{figure}}

\begin{proposition} \label{construct}(cf. Proposition 2.10 in \cite{papera}) For a fragment in $\Gamma(\mbox{\sc x,y})$ there exists a {\sc csc} path in $\Gamma(\mbox{\sc x,y})$ having their circular components of length less than $\pi$.
 \end{proposition}

\begin{definition} The path $\beta$ constructed in Proposition \ref{construct} is called a {\it replacement path} (see Figure \ref{figsmallhomot}).

\end{definition}

\begin{proposition}\label{lengthred}(cf. Proposition 3.3 in \cite{papera}) Given $\mbox{\sc x,y} \in T{\mathbb R}^2$. A $\it cs$ path in $\Gamma(\mbox{\sc x,y})$ containing an admissible component as a sub path is bounded-homotopic to another $\it cs$ path in $\Gamma(\mbox{\sc x,y})$ with less complexity being the length of the latter at most the length of the former.
\end{proposition}

With a geometric approach the next result proves the well known characterisation due to Dubins (compare \cite{dubins 1}) of the minimal length bounded curvature paths in ${\mathbb R}^2$.

\begin{theorem} \label{embdub}(cf. Theorem 3.9 in \cite{papera}) Choose $\mbox{\sc x,y} \in T{\mathbb R}^2$. The minimal length bounded curvature path in $\Gamma(\mbox{\sc x,y})$ is either a {\sc ccc} path having its middle component of length greater than $\pi $ or a {\sc csc} path where some of the circular arcs or line segments can have zero length.
\end{theorem}

Then in \cite{paperb} we extended our ideas in \cite{papera} to prove the following result which classifies the minimal length elements in connected components of bounded curvature paths in ${\mathbb R}^2$.

\begin{theorem} \label{singudub}(cf. Theorem 4.4 in \cite{paperb}) Given $\mbox{\sc x,y} \in T{\mathbb R}^2$ and $n\in {\mathbb Z}$. Then the minimal length bounded curvature path in $\Gamma(n)$ for $n\in {\mathbb Z}$ must be of the form:

\begin{itemize}
\item {\sc csc} or {\sc ccc}.
\item Symmetric $\mbox{\sc c}^\chi \mbox{\sc s}\mbox{\sc c}$ or $\mbox{\sc c}^{\chi}\mbox{\sc c} \mbox{\sc s}\mbox{\sc c}$.
\item Skew $\mbox{\sc c}^{\chi}\mbox{\sc s}\mbox{\sc c}$ or  $\mbox{\sc c}\mbox{\sc s}\mbox{\sc c}^{\chi}$.
\item $\mbox{\sc c}^\chi \mbox{\sc c}\mbox{\sc c}$ or $\mbox{\sc c} \mbox{\sc c}^{\chi}\mbox{\sc c}$.
 \end{itemize}
Here $\chi$ is the minimal number of crossings for paths in $\Gamma(n)$. In addition, some of the circular arcs or line segments may have zero length.
\end{theorem}

\section{Homotopies Preserving the Bounded Curvature Property}\label{section3}

In previous work we found the minimal length elements in spaces of bounded curvature paths, taking into consideration the possible configurations for the endpoint condition and winding number, compare \cite{papera} and \cite{paperb}. Then in \cite{paperc} we proved that bounded curvature paths in $\Delta(\Omega)$ are not bounded-homotopic to bounded curvature paths in $\Delta'(\Omega)$. Using properties of $\Omega$, we proved that paths in $\Delta(\Omega)$ have length bounded above. A crucial step in the classification of the homotopy classes in $\Gamma(\mbox{\sc x,y})$ is a continuity argument preserving the bounded curvature property for homotopies between general bounded curvature paths. The main result of this work is to determine the number of connected components of each space of bounded curvature paths, with the possible configurations of the adjacent circles (proximity conditions) and winding number.

Given a bounded curvature path and a fragmentation, the idea is to continuously deform each fragment onto its replacement path; these deformations will be considered depending on the component of the replacement path to which we are deforming the fragment, see Figure \ref{figsmallhomot}.

{ \begin{figure} [[htbp]
 \begin{center}
\includegraphics[width=.7\textwidth,angle=0]{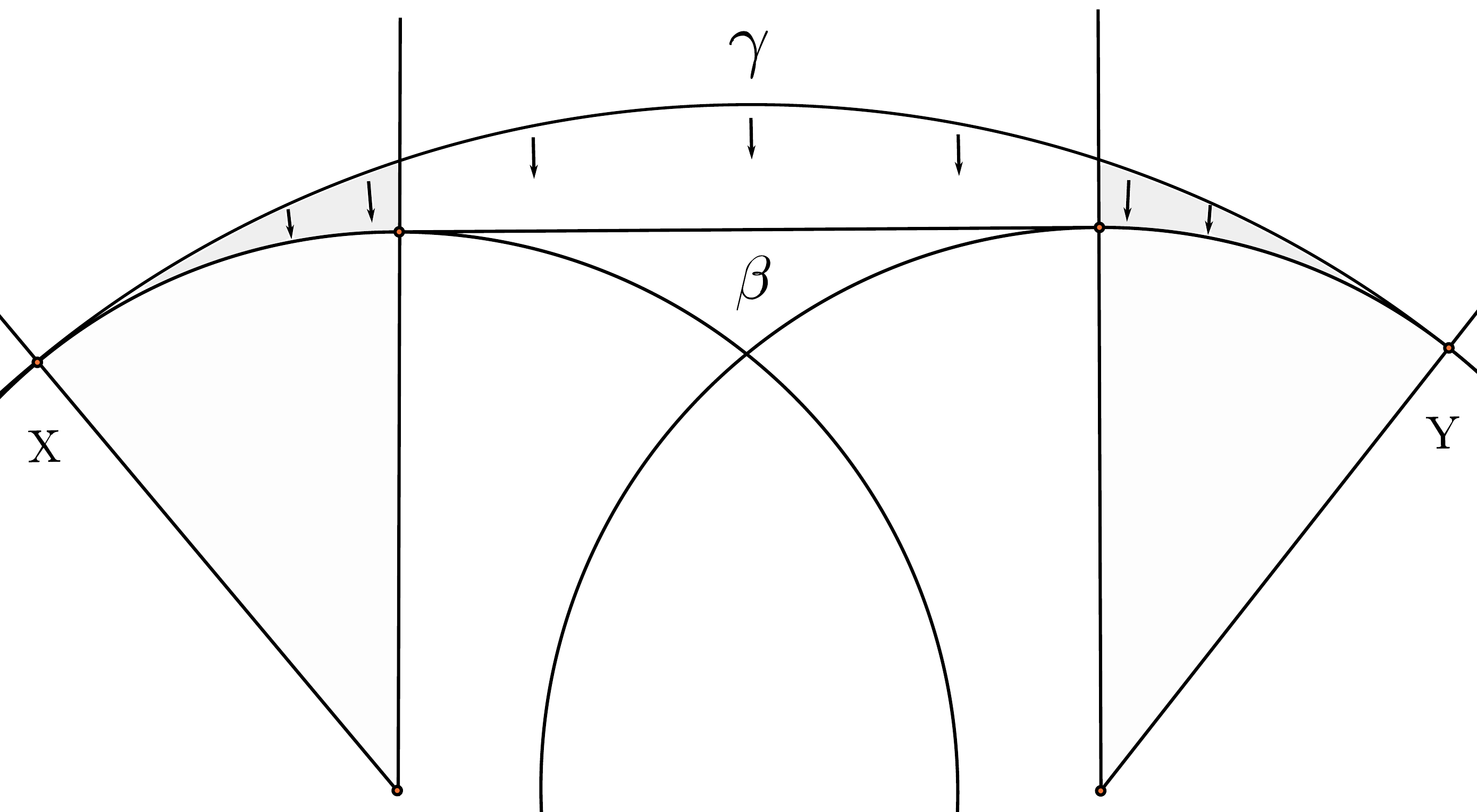}
\end{center}
\caption{An illustration of a fragment $\gamma$ and its replacement path $\beta$. The arrows suggest continuous deformation.}
\label{figsmallhomot}
\end{figure}}

In order to facilitate the exposition we denote homotopies by $H(p,t)$ rather than ${\mathcal H}_t(p)$ as we did in our previous works.

For a path $\gamma(t)=(u(t),v(t))$, we define the radial homotopy as:

$$R(p,t)=(1-p +\frac{p}{||(u(t),v(t))||})(u(t),v(t))$$

and the orthogonal homotopy as:

$$O(p,t)=(1-p)(u(t),v(t)) +p(0,v(t))$$

In order to ensure that the curvature is always bounded we will develop an argument relying on a second order approximation on sufficiently small fragments. As a consequence, we will have control of the local variation of the turning map applied in between these small fragments. In this fashion, we ensure that sufficiently small fragments mapped under the radial or orthogonal homotopy preserve the curvature bound. The wedge regions in Figure \ref{figsmallhomot} illustrate where to apply the radial homotopy to $\gamma$. We apply the orthogonal homotopy in the middle region in Figure \ref{figsmallhomot}.

We presented some simple operations performed on $cs$ paths (piecewise constant bounded curvature paths). In general, we need a continuity argument preserving the bounded curvature property for homotopies between bounded curvature paths that do not have piecewise constant bounded curvature.

Next, we  consider the curve,

 \begin{equation} {c}(\theta)=(x+\cos\theta, \sin \theta)\label{eqn1}\end{equation}
  for some $x\geq 0$ and $\theta \in [-\epsilon,\epsilon]$ with $\epsilon>0$. Also we will need the function $h: [0,1]\to{\mathbb R} $ where,
  \begin{equation} h(p)= 1-p +\frac{p}{a}\label{eqn2}\end{equation}
with $a>0$.
Putting definitions (\ref{eqn1}) and (\ref{eqn2}) together and setting $a=||(x+\cos\theta, \sin \theta)||$, we define the two variable map $R: [0,1]\times[-\epsilon,\epsilon]\to {\mathbb R}^2$ as,
 $$R(p,\theta)=h(p){c}(\theta)$$
So,
 \begin{equation} R(p,\theta)=(1-p +\frac{p}{||(x+\cos\theta, \sin \theta)||})(x+\cos\theta, \sin \theta) \label{preprehomo}\end{equation}
Next we consider second order Taylor approximations to $R$ fixing $p$, about $\theta=0$. The symbol $\approx$ represents equivalence up to second order Taylor approximation.

First note that, $$(x+\cos\theta, \sin \theta)\approx (x+1-\frac{\theta^2}{2},\sin \theta)\approx (x+1-\frac{\theta^2}{2},\theta)$$
On the other hand,
 $$||(x+\cos\theta, \sin \theta)||^2\approx ((x+1-\frac{\theta^2}{2})^2+\theta^2)\approx ((x+1)^2-x \theta^2)$$
So,
$$\frac{1}{||(x+\cos\theta, \sin \theta)||} \approx \frac{1}{\sqrt{(x+1)^2-x \theta^2}} =\frac{1}{(x+1)} \frac{1}{\sqrt{1-\frac{x\theta^2}{(x+1)^2}}} $$
Again, up to second order Taylor approximation,
$$\frac{1}{(x+1)} \frac{1}{\sqrt{1-\frac{x\theta^2}{(x+1)^2}}} \approx \frac{1}{(x+1)} \bigg( 1+\frac{1}{2}\frac{x}{(x+1)^2}\theta^2\bigg)$$
Next we replace the map $R: [0,1]\times[-\epsilon,\epsilon]\to {\mathbb R}^2$ by a second order Taylor approximation:
\begin{equation} H(p,\theta)=\bigg(1-p + \frac{p}{(x+1)} ( 1+\frac{1}{2}\frac{x}{(x+1)^2}\theta^2)\bigg)(x+1-\frac{\theta^2}{2},\theta)\label{prehomo}\end{equation}
By letting $u=x+1$ we obtain:
\begin{equation} H(p,\theta)=\bigg(1-p + \frac{p}{u} ( 1+\frac{1}{2}\frac{x}{u^2}\theta^2)\bigg)(u-\frac{\theta^2}{2},\theta)\label{prehomou}\end{equation}
After multiplying and then applying second order Taylor approximation to equation (\ref{prehomou}) we obtain:
\begin{equation} H(p,\theta)=\bigg(\theta^2(-\frac{1}{2}+\frac{p}{2}+\frac{px}{2u^2}-\frac{p}{2u})+u-pu+p, \theta(1-p+\frac{p}{u})\bigg)\label{homoutet}\end{equation}
By considering the substitution: $$w_\theta=\theta(1-p+\frac{p}{u})$$
We rewrite equation (\ref{homoutet}) as:
\begin{equation} H(p,\theta)=\bigg(\frac{w_\theta^2}{(1-p+\frac{p}{u})^2}(-\frac{1}{2}+\frac{p}{2}+\frac{px}{2u^2}-\frac{p}{2u})+u-pu+p, w_\theta\bigg)\label{homoutetw}\end{equation}
and consider:
$$A_p=1-p+\frac{p}{u}\qquad\mbox{and}\qquad B_p=-\frac{1}{2}+\frac{p}{2}+\frac{px}{2u^2}-\frac{p}{2u}$$
to write,
\begin{equation} (\frac{w_\theta^2}{A_p^2}B_p+u-pu+p, w_\theta)=(\phi(w_\theta),w_\theta)\label{homoutetw}\end{equation}
Then, by considering the two components of $H$ as only function of $\theta$ we write:
\begin{equation} H(p,w_\theta)\approx(\phi(w_\theta),w_\theta)\label{homograph}\end{equation}

As in equation (\ref{homograph}) we can view $H$ as a real-valued function of a single
variable. Therefore, by considering $ \frac{\partial w_\theta}{\partial \theta}=1-p+\frac{p}{u}$, the curvature of $\phi(w_\theta)$ can be easily computed by using the curvature of a graph formula:
\begin{equation} \kappa_{\phi}(p,\theta)=\frac{\frac{\partial^2 \phi}{\partial \theta^2}}{({1+[ \frac{\partial \phi}{\partial \theta}]^2)^{\frac{3}{2}}}}\label{homocurv}\end{equation}
Applying the chain rule we obtain:
\begin{equation}  \frac{\partial \phi}{\partial \theta}=\frac{2B_pw_\theta}{A_p}\qquad\mbox{and}\qquad  \frac{\partial^2 \phi}{\partial \theta^2}=2B_p \label{homocuvderi}\end{equation}
So that,
\begin{equation} \kappa_{\phi}(p,\theta)=\frac{2B_p}{(1+[2B_p\theta]^2)^{\frac{3}{2}}}\label{homocurvfirst}\end{equation}
An immediate evaluation in equation (\ref{homocurvfirst}) shows that $\kappa_{\phi}(0,0)=-1$. Also, a simple computation leads to:
\begin{equation}\frac{\partial \kappa_\phi}{\partial p}(0,0)=1+\frac{x}{2(x+1)}+\frac{x}{2(x+1)^2} \label{homocurvsecond}\end{equation}
Since $x\geq 0$ we conclude that:
\begin{equation}\frac{\partial \kappa_\phi}{\partial p}(0,0)>0 \label{homocurvsecond}\end{equation}

In addition, $\frac{\partial \kappa_\phi}{\partial p}(p,\theta)$ is clearly positive in a neighborhood of $(0,0)$. So, there exists $\epsilon>0$ with $|\theta|<\epsilon$ such that,
$$\frac{\partial \kappa_\phi}{\partial p}(0,0)>0$$

An identical argument shows that the second order approximation to a bounded curvature path around the point $\gamma(t)$ with $\kappa_{\phi}(0,0)=1$ is such that: There exists $\epsilon>0$ with $|\theta|<\epsilon$ such that,
$$\frac{\partial \kappa_\phi}{\partial p}(0,0)<0$$

 So, we have proved the following result.

\begin{lemma}\label{homotopycurvature} For the radial homotopy, there exists $\epsilon>0$ with $|\theta|<\epsilon$ such that if:
\begin{itemize}
\item $\kappa_{\phi}(0,0)=-1$ then $\frac{\partial \kappa_\phi}{\partial p}(0,0)>0$.
\item $\kappa_{\phi}(0,0)=1$ then $\frac{\partial \kappa_\phi}{\partial p}(0,0)<0$.
\end{itemize}
\end{lemma}

In other words, when applying the radial projection (see equation \ref{preprehomo}) to the second order approximation of $\gamma$ at $t\in I$ (for a small variation of the angle $\theta$ and a small variation of the homotopy parameter $p$) the curvature increases for the first statement in Lemma \ref{homotopycurvature} and decreases for the second statement in Lemma \ref{homotopycurvature}.

\begin{corollary}\label{homotopycurvaturecoro} For the orthogonal homotopy, there exists $\epsilon>0$ with $|\theta|<\epsilon$ such that if:
\begin{itemize}
\item $\kappa_{\phi}(0,0)=-1$ then $\frac{\partial \kappa_\phi}{\partial p}(0,0)>0$.
\item $\kappa_{\phi}(0,0)=1$ then $\frac{\partial \kappa_\phi}{\partial p}(0,0)<0$.
\end{itemize}
\end{corollary}

\begin{proof}  By applying a similar procedure as the one in Lemma \ref{homotopycurvature} the result follows. We leave the details to the reader.
\end{proof}

\begin{remark}\label{remainsbcp}   \hfill
\begin{itemize}

\item Let $\gamma$ be a fragment, denote $\gamma(t)=z$ and $\gamma'(t)=Z$. Let $L$ be the line  through $z$ orthogonal to $Z$. Assume the point $(x,0)$ is the center of $\mbox{\sc C}_ r(\mbox{\sc z})$ and let $o$ be the origin of a local coordinate system with $\gamma(t)$ lying on the positive $x$-axis and $x>0$. The dashed circle has center $o$, radius one and is used only as reference (see Figure \ref{figwedges0}).

\item Consider $\epsilon>0$ according to Lemma  \ref{homotopycurvature}. Recall that the second order approximation to $\gamma$ at $t\in I$ (locally) corresponds to a unit radius circle, then (locally) we can consider $\mbox{\sc C}_ r(\mbox{\sc z})$ to be the second order approximation at $\gamma(t)$ in case the curvature at $\gamma(t)$ is negative, or $\mbox{\sc C}_ l(\mbox{\sc z})$ in case the curvature at $\gamma(t)$ is positive. In addition, recall that by  Proposition \ref{r1r3pr} the region ${\mathcal R}(\mbox{\sc z})$ contains $\gamma$ (see Figure \ref{figleavreg}). We could flow the region ${\mathcal R}(\mbox{\sc z})$ using the radial (orthogonal) homotopy. Using similar ideas to Lemma \ref{homotopycurvature}, since $\gamma$ remains in ${\mathcal R}(\mbox{\sc z})$ during the flow, this could be used to show the curvature of $\gamma$ remains bounded by 1 as desired.

\end{itemize}
\end{remark}

{ \begin{figure} [[htbp]
 \begin{center}
\includegraphics[width=.5\textwidth,angle=-1]{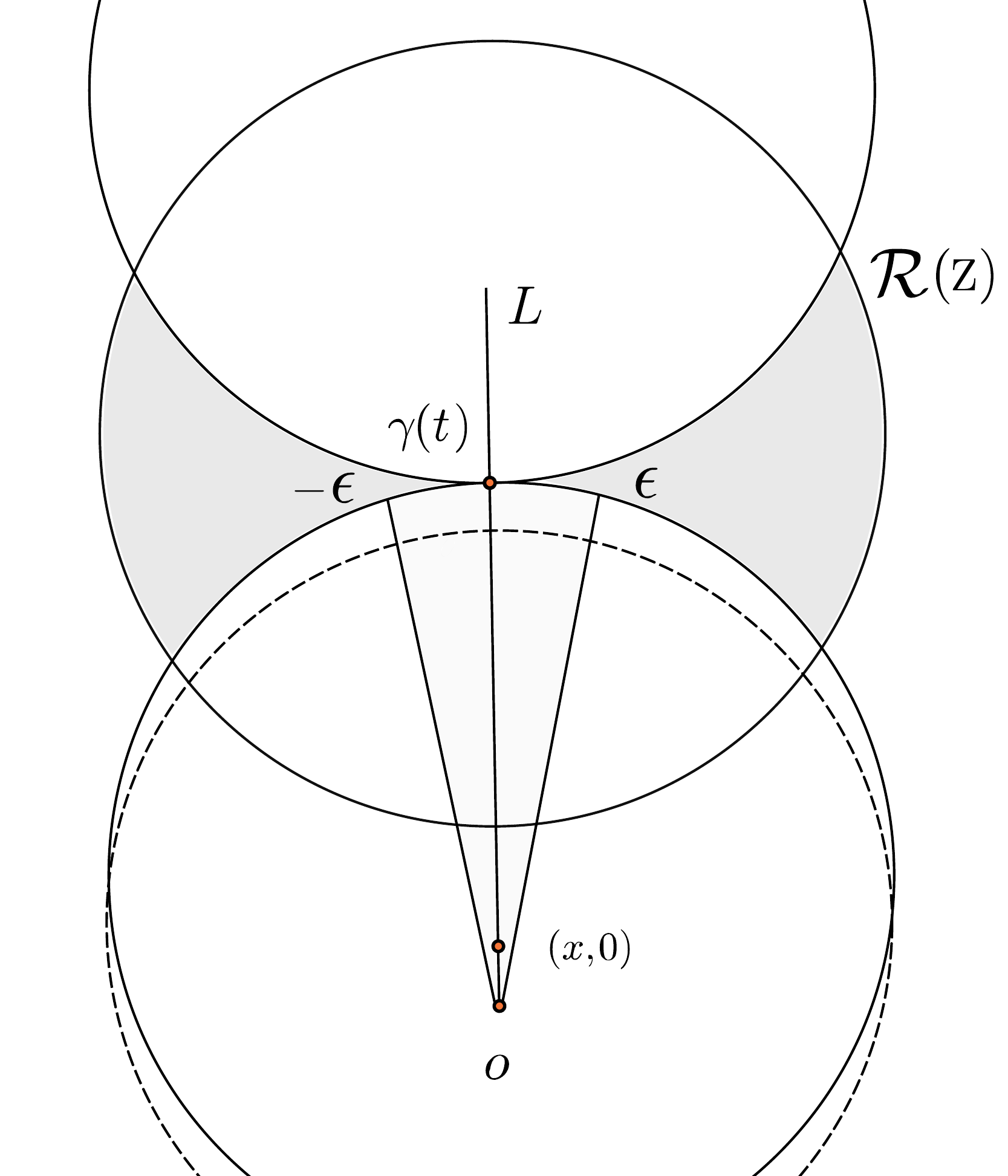}
\end{center}
\caption{Description of notation in Remark \ref{remainsbcp}. The right adjacent circle of $\gamma(t)$ has center at $(x,0)$, the dashed circle has center $o$ and the shaded region corresponds to ${\mathcal R}(\mbox{\sc z})$. We exaggerate the size of the arc $|\theta|<\epsilon$ to facilitate the exposition, see Figure \ref{figwedges1} for a blow up around $\gamma(t)$.}
\label{figwedges0}
\end{figure}}

{ \begin{figure} [[htbp]
 \begin{center}
\includegraphics[width=.9\textwidth,angle=0]{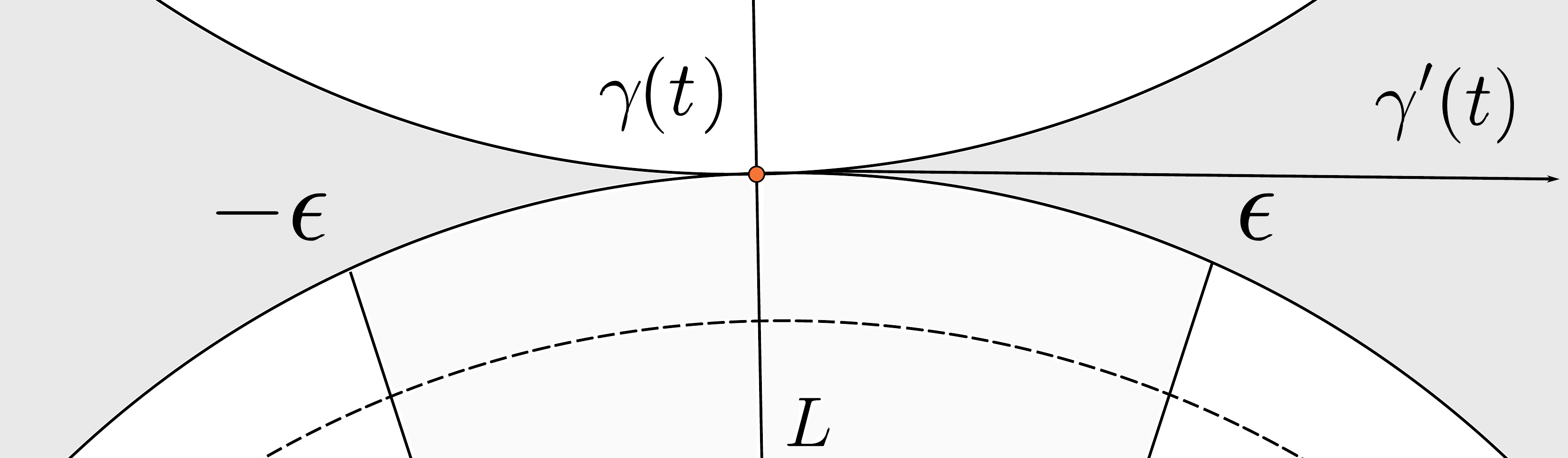}
\end{center}
\caption{A blow up around $\gamma(t)$ in Figure \ref{figwedges0}. Note that $\gamma'(t)$ meets orthogonally $L$, also note that the path $\gamma$ travels from left to right as usual.}
\label{figwedges1}
\end{figure}}

It is crucial to know when the radial (orthogonal) homotopy can be applied so that the curvature remains bounded by 1.

\begin{lemma}\label{conticurv} Given $\epsilon>0$ corresponding to the length of a bounded curvature path $\gamma$, there exists $\delta>0$ such that the curvature at $\gamma(t)$ for $t\in (t-\delta,t+\delta)$ remains bounded by 1.
\end{lemma}

\begin{proof} Let $\gamma$ be a bounded curvature path of length $\epsilon>0$ and consider a local coordinate system as in Remark \ref{remainsbcp}. By Lemma \ref{homotopycurvature} we know that the curvature of $\mbox{\sc C}_ r(\mbox{\sc z})$ and $\mbox{\sc C}_{ l}(\mbox{\sc z})$ remain bounded under the application of $R$ for a sufficiently small homotopy parameter.  Due to continuity of the derivative, there is a neighborhood of $t\in I$ where the curvature increases if  $\kappa_{\phi}(0,0)=-1$ (or decreases if $\kappa_{\phi}(0,0)=1$) i.e., there exists $\delta>0$ such that the curvature at $\gamma(t)$ for $t\in (t-\delta,t+\delta)$ remains bounded by 1.
\end{proof}

\begin{lemma} \label{finefrag} A bounded curvature path admits a fragmentation with fragments satisfying Lemma \ref{conticurv}.
\end{lemma}
\begin{proof} Immediate since bounded curvature paths are $C^1$.
\end{proof}

{ \begin{figure} [[htbp]
 \begin{center}
\includegraphics[width=.9\textwidth,angle=0]{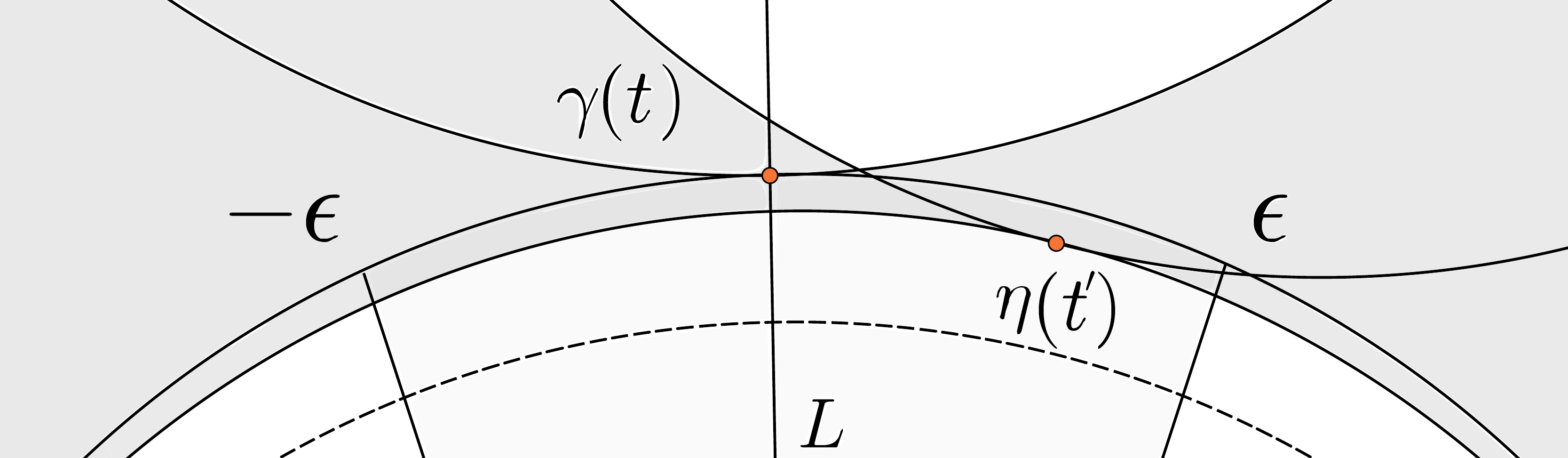}
\end{center}
\caption{The variation of the angle of the derivative vector between $\gamma$ and an arbitrarily close bounded curvature path $\eta$ in the arc $|\theta|<\epsilon$ in the dashed circle can be chosen to be less than $\delta$, compare with Figure \ref{figwedges1}.}
\label{figwedges2}
\end{figure}}

\begin{proposition}\label{homotopyfragment} A fragment satisfying Lemma \ref{conticurv} is bounded-homotopic to its replacement path.
\end{proposition}
\begin{proof} Consider a fragment satisfying Lemma \ref{conticurv}. We apply Proposition \ref{construct} to construct a replacement path. By Lemma \ref{homotopycurvature} and Corollary \ref{homotopycurvaturecoro}, in such a fragment, the radial and orthogonal homotopy preserve the bounded curvature property. Since the fragment and the replacement path are $C^1$ close, the construction of a homotopy between them is a piecewise combination of two radial projections (sides) and a orthogonal projection (middle) onto the replacement path (see Figure \ref{figsmallhomot}). We leave the details to the reader.\end{proof}

\begin{theorem} A bounded curvature path is bounded-homotopic to a cs path.
\end{theorem}
\begin{proof} Consider $\gamma \in \Gamma(\mbox{\sc x,y})$ together with a fragmentation of $\gamma$ with fragments satisfying Lemma \ref{conticurv}. Then by applying Proposition \ref{homotopyfragment} to each fragment the result follows.
\end{proof}

A step towards the classification of homotopy classes of bounded curvature paths is a proof of Theorem \ref{embdub} (refer to \cite{papera}) and Theorem \ref{singudub} (refer to \cite{paperb}) for homotopy classes.

\begin{theorem}\label{mininhomot} A homotopy class of bounded curvature paths contains a minimal length path.
\end{theorem}
\begin{proof} Consider $\gamma \in \Gamma(\mbox{\sc x,y})$ together with a fragmentation of $\gamma$ with fragments satisfying Lemma \ref{conticurv}. By applying Proposition \ref{homotopyfragment} to all the fragments we obtain a homotopy of bounded curvature paths between $\gamma$ and a $cs$ path. By reducing the complexity of the $cs$ path applying Proposition \ref{lengthred} a finite number of times we obtain the desired minimal length path.
\end{proof}

\section{Homotopies Between Dubins Paths}

 In Theorem 110 in \cite{thesisayala} we proved (under our conventions about closure paths, see Remark \ref{convlambda}) that the global minimal length bounded curvature paths lie $\Gamma(\mbox{\rm P},k)$ for  $k=-1,0,1$. Then in Lemma 113 in \cite{thesisayala} we proved that if $\gamma\in \Delta(\Omega)$ with closure path $\lambda$ then $W_\lambda(\gamma)=1$ or $W_\lambda(\gamma)=-1$.

Next we give examples illustrating the existence of minimal length elements in $\Gamma(\mbox{\rm P},k)$.

{ \begin{figure} [[htbp]
 \begin{center}
\includegraphics[width=0.7\textwidth,angle=0]{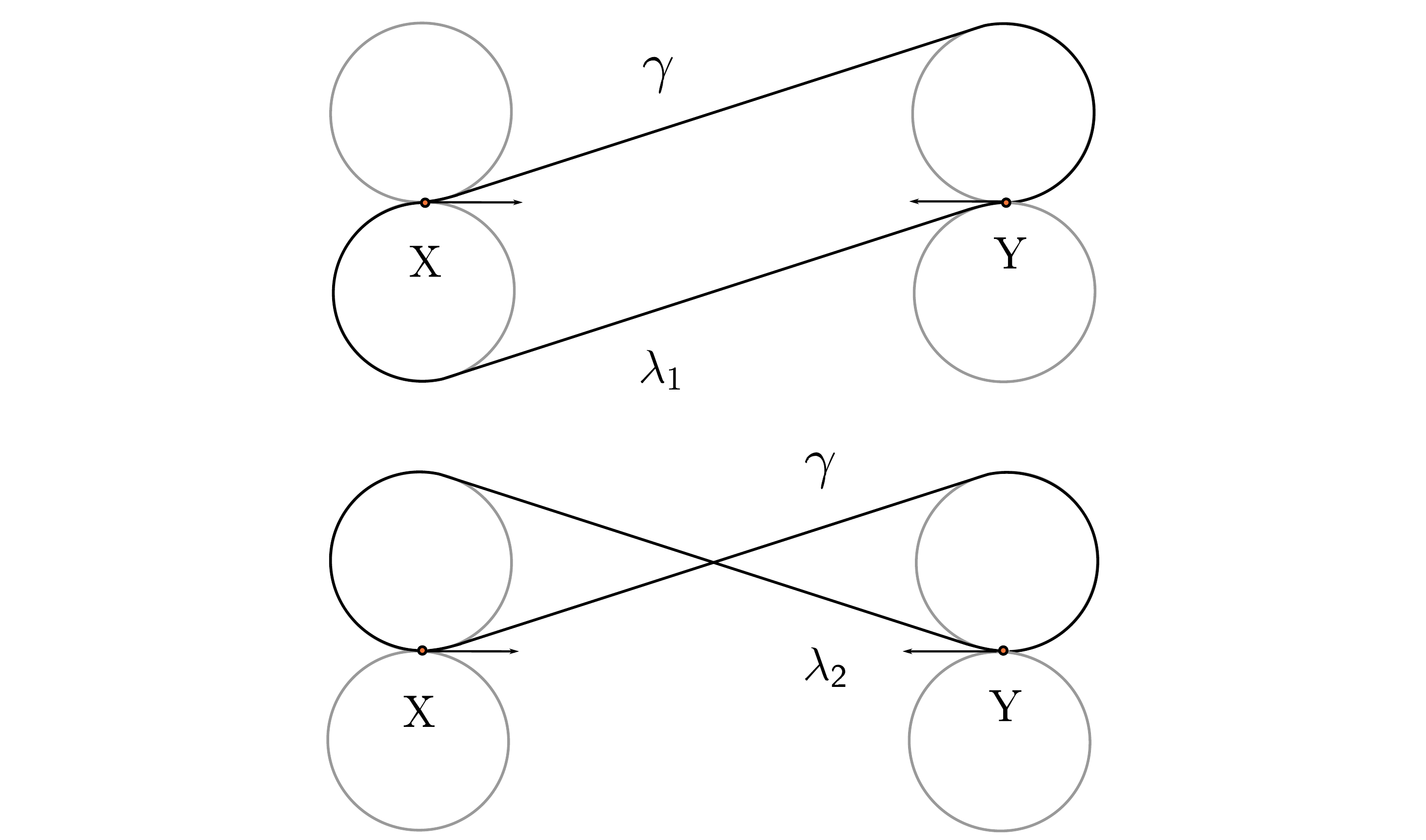}
\end{center}
\caption{Top: An example of an element in $\Gamma(\mbox{\rm A},1)$. Bottom: An example of an element in $\Gamma(\mbox{\rm A},0).$  }
\label{figwinlamb}
\end{figure}}

 In Figure \ref{figwinlamb} we show minimal length elements in $\Gamma(\mbox{\rm A},k)$ for $k=0,1$. An example for $\Gamma(\mbox{\rm A},-1)$ is obtained by reflecting the case for $k=1$ with respect to the line $y=0$. In Figure \ref{figminbd} we see cases of minimal length elements in $\Gamma(\mbox{\rm B},k)$ for $k=-1,0,1$ and in Figure \ref{figminbd} below we see cases of minimal length elements in $\Gamma(\mbox{\rm C},k)$ for $k=-1,0,1$. The existence of minimal length elements in $\Gamma(\mbox{\rm D},1)$ and $\Gamma(\mbox{\rm D},-1)$ is also immediately verified. Of course, the minimality of $\gamma$ and $\lambda$ is trivial by simply comparing the six possible minimal paths given by  Theorem \ref{embdub} for the $\mbox{\sc x,y} \in T{\mathbb R}^2$ from {\sc x} to {\sc y} and from {\sc y} to {\sc x} respectively.

{ \begin{figure} [[htbp]
 \begin{center}
\includegraphics[width=.6\textwidth,angle=0]{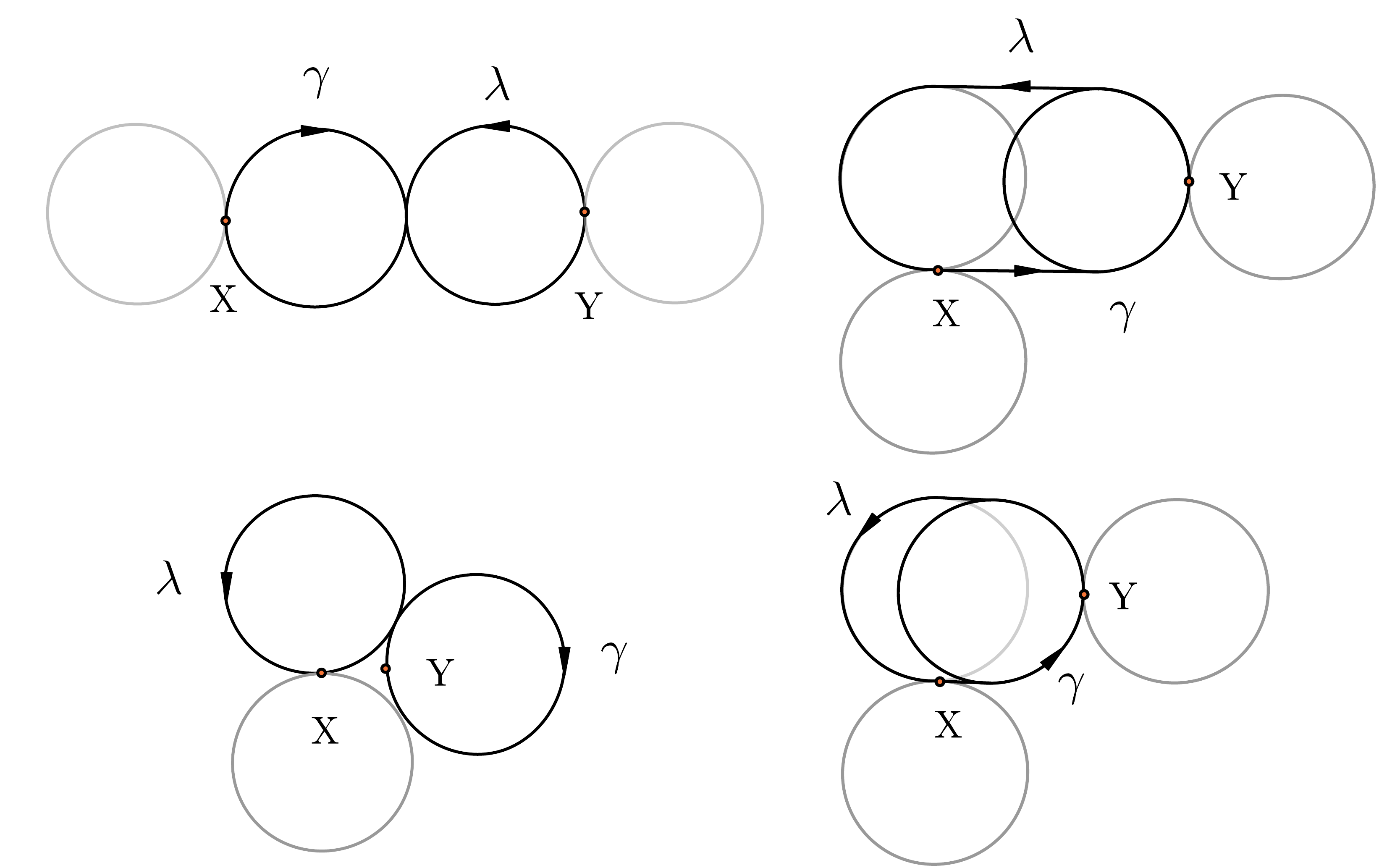}
\end{center}
\caption{Examples of minimal paths in $\Gamma(\mbox{B},k)$ and $\Gamma(\mbox{C},k)$ with $k=-1,0,1$. The cases where $k=1$ are obtained by reflecting the cases for $k=-1$ with respect to the line $y=0$. The arrows imply path orientation.}
 \label{figminbd}
\end{figure}}

A curious fact is that sometimes the minimal length path in the spaces of paths not in $\Omega$ is a winding number zero path. Even more, there are examples where there are two minimal length elements in the spaces of paths not in $\Omega$ lying in different homotopy classes, compare Figure \ref{figminext}.

\begin{remark}\label{minuniq} Given $\mbox{\sc x,y} \in T{\mathbb R}^2$. Then after choosing the closure path $\lambda$ (which may not be unique, but we fix a preferred one) the minimal length path from {\sc x} to {\sc y} may not be unique. These minimal paths can lie in the same homotopy class, or be in different homotopy classes. In general, we refer to {\it uniqueness of minimal length paths} {\it up to  operations on cs paths}.
\end{remark}

{ \begin{figure} [[htbp]
 \begin{center}
\includegraphics[width=.7\textwidth,angle=0]{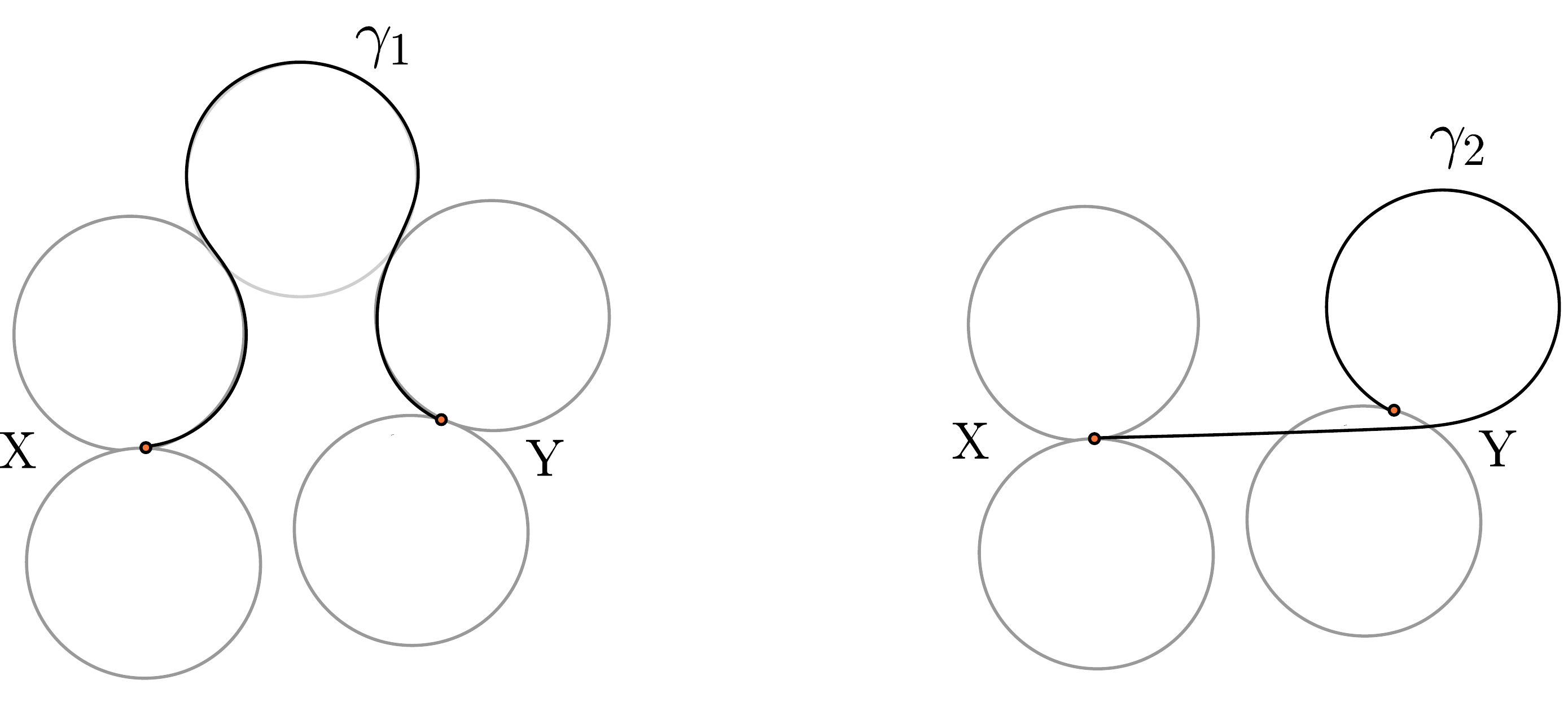}
\end{center}
\caption{The paths $\gamma_1$ and $\gamma_2$ are candidates for being minimal length elements in the spaces of paths not in $\Omega$. It is not hard to see that these paths have winding number one and zero respectively.}
 \label{figminext}
\end{figure}}

Next we enumerate some {\it simple facts} about Dubins paths. These facts are related with the proximity condition satisfied by the associated endpoint conditon. Their verification is immediate and we leave this to the reader.

\begin{remark}\label{minhomot}For $\mbox{\sc x,y} \in T{\mathbb R}^2$ and $k=-1,0,1$ we have:
\begin{enumerate}
\item  The only Dubins paths in $\Gamma(\mbox{\rm A},k)$ are the {\sc csc} paths.
\item The spaces $\Gamma(\mbox{\rm B},k)$ contain {\sc csc} and {\sc ccc} paths. But a minimal Dubins path must be a {\sc csc} path.
\item The spaces $\Gamma(\mbox{\rm C},k)$ contain {\sc csc} and {\sc ccc} paths. A minimal Dubins path can be a {\sc csc} path or a {\sc ccc} path.
\item If a minimal Dubins path is {\sc ccc} then it must lie in a space of type $\Gamma(\mbox{\rm C},k)$.
\item The spaces $\Gamma(\mbox{\rm D},k)$ contain {\sc csc} and {\sc ccc} paths. But a minimal Dubins path must be a {\sc csc} path.
\end{enumerate}
\end{remark}

Next we describe some explicit homotopies between bounded curvature paths. Our aim is not to make a long list of possibilities, we rather prefer to discuss the less obvious cases.

\begin{proposition}\label{skew} An {\sc rsl} Dubins path can sometimes be deformed into a {\sc lsr} Dubins path.
\end{proposition}

\begin{proof} Here we diagrammatically illustrate our claim in Figure \ref{figskewdef} from left to right. In the first step we start with an {\sc rsl} Dubins path and select a point in its middle component. Next we apply an operation of {\it type I} to the middle component to obtain a figure 8 shape. In the third step we translate these oppositely oriented loops to lie in the circles $\mbox{\sc C}_l(\mbox{\sc x})$ and $\mbox{\sc C}_r(\mbox{\sc y})$ to obtain in step five the desired {\sc lsr} Dubins path. \end{proof}

{ \begin{figure}[htbp]
 \begin{center}
\includegraphics[width=1\textwidth,angle=0]{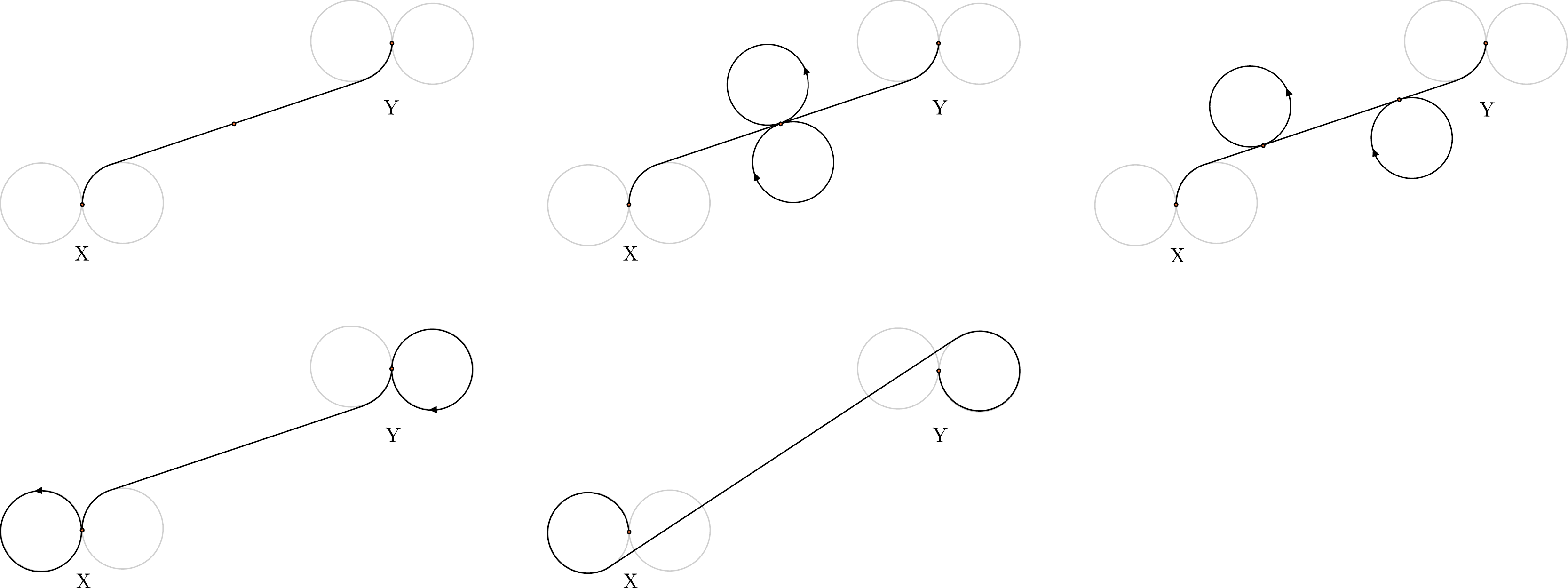}
\end{center}
\caption{Select a point in a {\sc rsl} path to apply an operation of type I and then a translation we homotope a {\sc rsl} into a {\sc lsr} path.}
\label{figskewdef}
\end{figure}}

The next result gives an explicit homotopy of bounded curvature paths between Dubins paths of type {\sc ccc} and {\sc csc} under proximity condition {\rm B} and between Dubins paths of type {\sc ccc} satisfying proximity condition {\rm D}. We alert the reader that the next construction can be performed with fewer operations. However, we present the sequence of moves in Figure \ref{fighomotopy} in order to make the exposition clearer.

\begin{proposition}\label{homotopy2}\rm The paths $\gamma_1$ and $\gamma_2$ in Figure \ref{figcomphop} are in the same homotopy class.
\end{proposition}

{ \begin{figure}[htbp]
 \begin{center}
\includegraphics[width=.75\textwidth,angle=0]{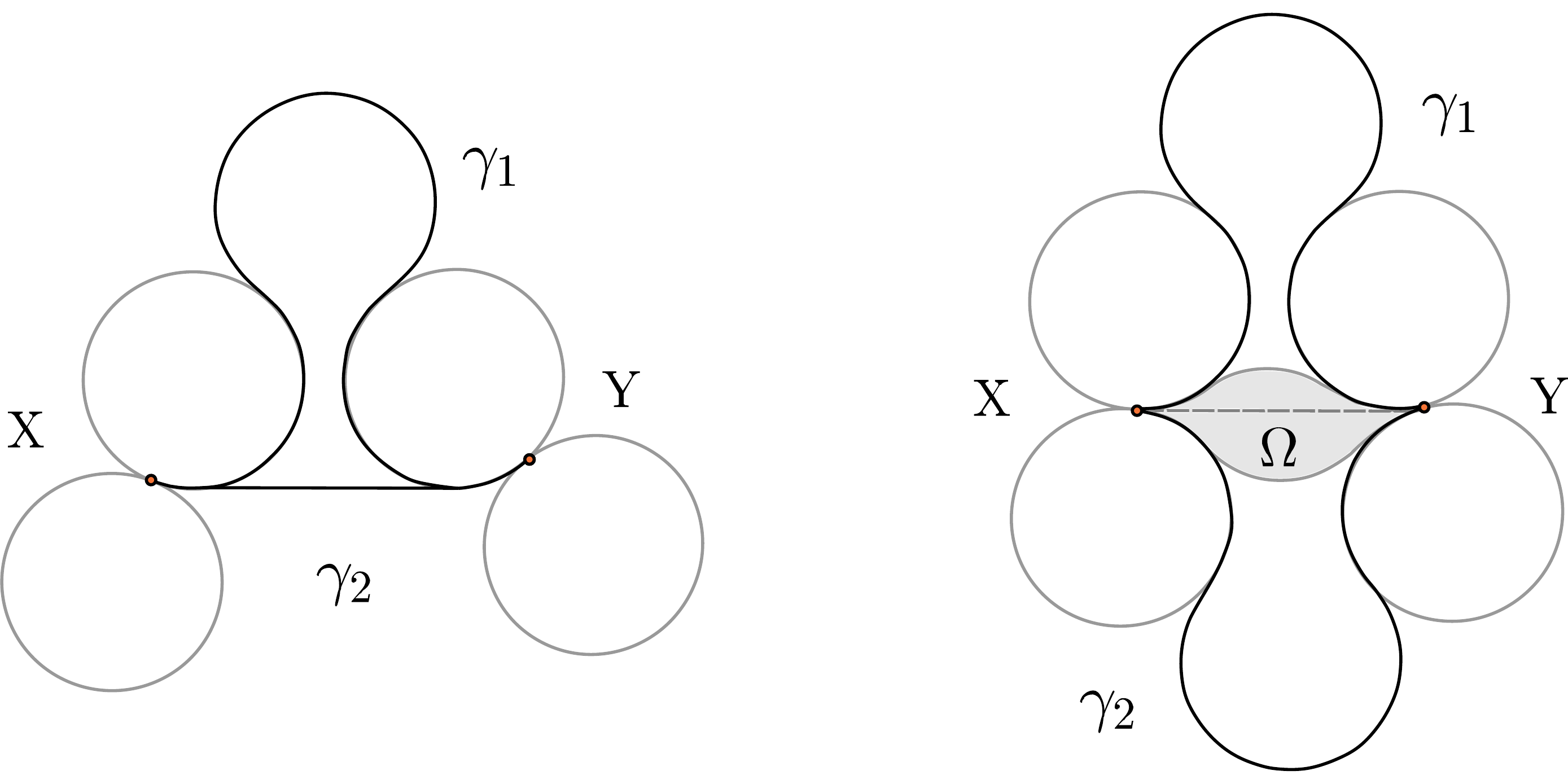}
\end{center}
\caption{Left: The paths $\gamma_1,\gamma_2\in \Gamma(\mbox{B},1)$ are bounded-homotopic. Right: The paths $\gamma_1$ and $\gamma_2$ are bounded-homotopic. By Corollary \ref{mainresultp1}  we have that $\gamma_1,\gamma_2\notin  \Delta(\Omega)$.}
\label{figcomphop}
\end{figure}}

\begin{proof} We first prove that there is a bounded curvature homotopy between the paths $\gamma_1$ and $\gamma_2$ satisfying condition {\rm B} giving an explicit construction of the homotopy, see left illustration in Figure \ref{figcomphop}.
We label each step in Figure \ref{fighomotopy} from left to right starting with 1 and finishing at 10.
In step 1 we start with the path $\gamma_1$ of type {\sc lrl} as given at the left in Figure \ref{figcomphop}. In step 2 we move the {\sc rl} part of $\gamma_1$ towards the left. In step 3 we obtain a {\sc lrlsl} path. In step 4 we rotate the upper circle, corresponding to the {\sc r} section in the {\sc lrlsl} path, clockwise. In step 5 another {\sc lrlsl} path is obtained. In step 6 we rotate clockwise the first component of the path obtained in the last step. In steps 8 and 9, using the existence of parallel tangents, we unfold the figure 8 shape obtained in step 7 to obtain in step 10 a path bounded-homotopic to the path $\gamma_2$.
The same procedure may be applied to obtain a homotopy of bounded curvature paths between $\gamma_1$ and $\gamma_2$ satisfying condition {\rm D} (right side in Figure \ref{figcomphop}). In such a case, the path $\gamma_2$ in step 10 is of the same type as $\gamma_1$ since under condition {\rm D} we also have that $d(c_r(x),c_r(y))<4$.
\end{proof}

{ \begin{figure} [[htbp]
 \begin{center}
\includegraphics[width=1.0\textwidth,angle=0]{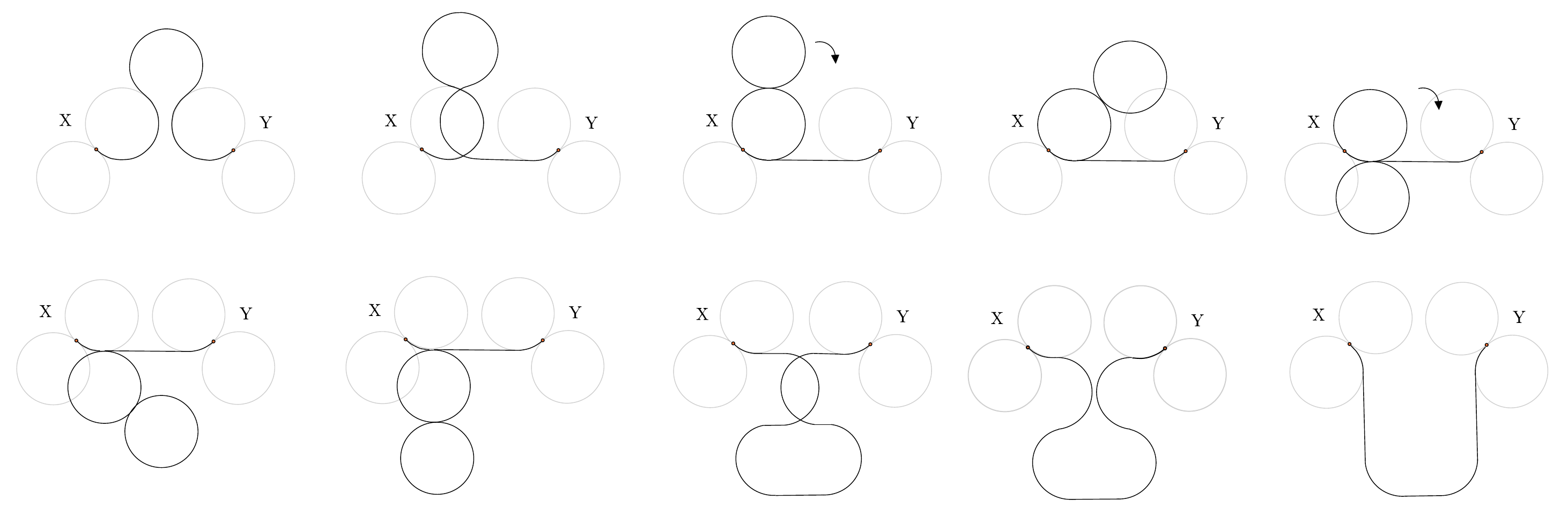}
\end{center}
\caption{A bounded curvature homotopy between the paths $\gamma_1$ and $\gamma_2$ at the left in Figure \ref{figcomphop}.}
 \label{fighomotopy}
\end{figure}}

\section{The Classification Theorem in $\Gamma(k)$}

Next we prove that $\Delta(\Omega)$ and $\Delta'(\Omega)$ are maximal path connected sets.
\begin{definition} A homotopy class is said to be free if it contains a free path. We denote by $[\gamma]$ the homotopy class of the path $\gamma$.
\end{definition}

\begin{theorem}\label{singlehomot} If $\mbox{\sc x},\mbox{\sc y}\in T{\mathbb R}^2$ carries a region $\Omega$. The spaces $\Gamma(\mbox{\rm  D},k)$ are such that:

\begin{itemize}
\item $\Delta(\Omega)$ is not a {\it free} homotopy class.
\item $\Delta'(\Omega)$ is a {\it free} homotopy class.
\end{itemize}
\end{theorem}

\begin{proof}  We first prove that $\Delta(\Omega)$ and $\Delta'(\Omega)$ each correspond to a homotopy class.  Consider a path $\gamma \in \Delta(\Omega)$ together with its homotopy class $[\gamma]$. By applying the procedure in Theorem \ref{mininhomot} to $\gamma$ we conclude that $[\gamma]$ must contain a minimal length element which must be is of type {\sc csc}. By virtue of Proposition 111 in \cite{thesisayala} the minimal length element is an element in $\Delta(\Omega)$ . Next consider a path $\delta \in \Delta(\Omega)$, different from $\gamma$. By applying Theorem \ref{mininhomot} to $\delta$ we conclude that $[\delta]$ must also contains a minimal length element. By Remark \ref{minuniq} the minimum length path is unique in its class. Therefore, $[\gamma]=[\delta]$. We conclude that $\Delta(\Omega)$ is a connected space.

To prove that $\Delta'(\Omega)$ corresponds to a homotopy class note that Corollary \ref {mainresultp1} establishes that paths in $\Delta(\Omega)$ and $\Delta'(\Omega)$ are not bounded-homotopic. Therefore, by applying the procedure in Theorem \ref{mininhomot} to $\gamma \in \Delta'(\Omega)$ we conclude that $[\gamma]$ must contain a minimal length element. Then via an analogous procedure as in the last paragraph we conclude that $\Delta'(\Omega)$ is connected.

By virtue of Theorem \ref{mininhomot} a homotopy class of bounded curvature paths contains a lower bound for the length of its elements. Then by applying Theorem \ref{insbound} we conclude that $\Delta(\Omega)$ is not a free homotopy class. In addition, note that it is trivial to construct paths in $\Delta'(\Omega)$ having arbitrary large length, showing that $\Delta'(\Omega)$ is a free homotopy class.
\end{proof}

The next result gives a classification of the homotopy classes for elements in $\Gamma(k)$ for any given endpoint condition.

\begin{theorem}\label{theoclassemb} Given $\mbox{\sc x},\mbox{\sc y}\in T{\mathbb R}^2$. Then:
\begin{itemize}
\item $\Gamma(\mbox{\rm  A},k)$, $\Gamma(\mbox{\rm  B},k)$ and $\Gamma(\mbox{\rm  C},k)$ are free homotopy classes. And,
\item $\Gamma(\mbox{\rm  D},k)$ consists of two homotopy classes:
\begin{enumerate}
\item $\Delta(\Omega)$  is a not a free homotopy class.
\item $\Delta'(\Omega)$  is a free homotopy class.
\end{enumerate}
 Or
\item $\Gamma(\mbox{\rm  D},k)$ correspond to a space of bounded curvature paths containing an isolated point consists of two homotopy classes: One corresponds to the isolated point. The other is a free homotopy class.
\end{itemize}
\end{theorem}
\begin{proof} By an identical argument as that given in Theorem \ref{singlehomot} but now for $\Gamma(\mbox{\rm  A},k)$, we conclude that $\Gamma(\mbox{\rm  A},k)$ corresponds to a single path component. By applying an operation of type II to the minimal element in $\Gamma(\mbox{\rm  A},k) $, we conclude that $\Gamma(\mbox{\rm  A},k)$ is a free homotopy class.
For $\Gamma(\mbox{\rm  B},k)$ we proceed identically as for $\Gamma(\mbox{\rm  A},k)$, just noticing that Proposition \ref{homotopy2} ensures that the path $\gamma_1$ (for condition B) can be deformed to the minimal length path $\gamma_2$ in Figure \ref{figcomphop}. A similar argument applies for $\Gamma(\mbox{\rm C},k)$. Therefore we conclude that $\Gamma(\mbox{\rm B},k)$ and $\Gamma(\mbox{\rm C},k)$ are free homotopy classes.
The proof of the second item was given in Theorem \ref{singlehomot}.

In \cite{thesisayala} we proved in Proposition 114 and Proposition 115 that:
\begin{itemize}
\item A bounded curvature path consisting of a single arc of a unit circle of length less than $\pi$ is a homotopy class consisting of an isolated point in ${\Gamma}(\mbox{\sc x,y})$.
\item A bounded curvature path consisting of a concatenation of two arcs of unit circles each of length less than $\pi$ is a homotopy class consisting of an isolated point in ${\Gamma}(\mbox{\sc x,y})$.

\end{itemize}
 In particular, such a unique element is not a free path.

Now, the existence of a minimal length path in the complement of the isolated point, is identical as for $\Delta'(\Omega)$ in Theorem \ref{singlehomot}. Then, we conclude that the complement of the isolated point is free homotopy class. Of course the isolated point is of is not a free homotopy class.
\end{proof}

 Since $\Gamma(\mbox{\rm  A},k)$, $\Gamma(\mbox{\rm  B},k)$, $\Gamma(\mbox{\rm  C},k)$ and $\Delta'(\Omega)$ are free homotopy classes it is easy to see that their elements are bounded-homotopic to bounded curvature paths with self intersections. By contrast, we have the following result.

\begin{theorem}\label{hcembedded} The spaces $\Delta(\Omega)$ correspond to homotopy classes of embedded paths.
\end{theorem}

\begin{proof} In Theorem \ref{theoclassemb} we proved that the spaces $\Delta(\Omega)$ are homotopy classes. In Corollary \ref{cannotsing} we proved that embedded bounded curvature paths in $\Omega$ cannot be made bounded-homotopic to bounded curvature paths with self intersections concluding the proof.
\end{proof}

\section{The Classification Theorem in $\Gamma(n)$}

The Graustein-Whitney theorem states that two planar closed curves with the same winding number may be deformed one into another by a regular homotopy (homotopy through immersions). Conversely, we cannot deform two closed curves of distinct winding numbers one into another by such a homotopy compare \cite{whitney}. From this observation we see that if two paths have different winding numbers, they are in different path components (see Corollary \ref{wcpm}).

Recall that $\Gamma(k)$ correspond to the space in $\Gamma(\mbox{\sc x,y})$ containing the global minimum of length. Next we prove that each of the spaces $\Gamma(n)$ for $n\neq k$ contains only one connected component.

\begin{theorem} The spaces $\Gamma(n)$ correspond to a single homotopy class for $n\neq k$.
\end{theorem}
\begin{proof} Consider a path $\gamma \in \Gamma(n)$ and a fragmentation whose fragments have length at most $\delta>0$ (small), see Proposition \ref{homotopyfragment}. Then apply Proposition \ref{homotopyfragment} to obtain a $cs$ in $[\gamma]$. Then proceed to bounded-homotope all the components of type ${\mathscr C}_3$ (if there is any) and then collapse all the loops to one of the adjacent circles as we did in Theorem \ref{singudub}. We apply Theorem \ref{mininhomot} to prove the existence of a minimal length element. Then by an identical argument as in Theorem \ref{singlehomot} we conclude that $\Gamma(n)$ corresponds to a single homotopy class. \end{proof}

We have proven the following result.

\begin{theorem} \label{classificationbcps}Given $\mbox{\sc x},\mbox{\sc y}\in T{\mathbb R}^2$ we have that:

\begin{equation}
  \Gamma(\mbox{\sc x},\mbox{\sc y}) =
       \bigcup_{\substack{ n\in {\mathbb Z}\\
                 \\
                  }}
         \Gamma(n)
\end{equation}

In particular, if $\mbox{\sc x},\mbox{\sc y}\in T{\mathbb R}^2$ satisfies condition {\rm D} we have that:

 $$\Delta(\Omega) \cup \Delta'(\Omega)=\Gamma(k)$$
 with,
 $$\Delta(\Omega) \cap \Delta'(\Omega)=\emptyset $$
  or, $\Gamma(k)$ corresponds to the union of a free homotopy class and an isolated point.
  \end{theorem}

\bibliographystyle{amsplain}
   
\end{document}